\documentclass[reqno]{amsart}
\usepackage{graphicx}
\usepackage{amssymb}
\usepackage{amsfonts}
\usepackage{hyperref}
\usepackage{mathrsfs}
\swapnumbers
\newtheorem{theorem}{Theorem}[subsection]
\newtheorem{lemma}[theorem]{Lemma}
\newtheorem{corollary}[theorem]{Corollary}

\newtheorem{proposition}[theorem]{Proposition}
\theoremstyle{definition}

\newtheorem{assumption}[theorem]{Assumption}
\newtheorem{remark}[theorem]{Remark}

\numberwithin{equation}{section}
\theoremstyle{plain}

\numberwithin{equation}{section} 
\numberwithin{figure}{section} 
\theoremstyle{plain}
\theoremstyle{plain}
\theoremstyle{remark}
\newtheorem*{acknowledgement*}{Acknowledgement}
\theoremstyle{example}

\newcommand{\cA}{{\mathcal A}}
\newcommand{\cB}{{\mathcal B}}
\newcommand{\cC}{{\mathcal C}}
\newcommand{\cD}{{\mathcal D}}
\newcommand{\cE}{{\mathcal E}}
\newcommand{\cF}{{\mathcal F}}
\newcommand{\cG}{{\mathcal G}}
\newcommand{\cH}{{\mathcal H}}

\newcommand{\cL}{{\mathcal L}}
\newcommand{\cM}{{\mathcal M}}
\newcommand{\cN}{{\mathcal N}}

\newcommand{\cS}{{\mathcal S}}

\newcommand{\cW}{{\mathcal W}}
\newcommand{\cX}{{\mathcal X}}

\newcommand{\cZ}{{\mathcal Z}}
\newcommand{\te}{{\theta}}

\newcommand{\Om}{{\Omega}}
\newcommand{\om}{{\omega}}
\newcommand{\ve}{{\varepsilon}}
\newcommand{\del}{{\delta}}
\newcommand{\Del}{{\Delta}}
\newcommand{\gam}{{\gamma}}
\newcommand{\Gam}{{\Gamma}}

\newcommand{\Sig}{{\Sigma}}
\newcommand{\sig}{{\sigma}}
\newcommand{\al}{{\alpha}}
\newcommand{\be}{{\beta}}
\newcommand{\ka}{{\kappa}}

\newcommand{\bbE}{{\mathbb E}}

\newcommand{\bbN}{{\mathbb N}}

\newcommand{\bbR}{{\mathbb R}}

\newcommand{\bbZ}{{\mathbb Z}}
\newcommand{\bbI}{{\mathbb I}}
\newcommand{\bbQ}{{\mathbb Q}}

\newcommand{\brF}{{\bar F}}

\begin{document}
\title[]{A functional CLT for nonconventional polynomial arrays}
\author{Yeor Hafouta \\
\vskip 0.1cm
Department of mathematics\\
The Ohio state University and the Hebrew University}

\thanks{ }
\dedicatory{  }
 \date{\today}

\footnotetext[1]{}
\maketitle
\markboth{Y. Hafouta}{A functional CLT for nonconventional polynomial arrays} 
\renewcommand{\theequation}{\arabic{section}.\arabic{equation}}
\pagenumbering{arabic}

\begin{abstract}\noindent
In this paper we will prove a functional central limit theorem (CLT) for random functions of the form 
\[
\cS_N(t)=N^{-\frac12}\sum_{n=1}^{[Nt]} F(\xi_{q_1(n,N)},\xi_{q_2(n,N)},...,\xi_{q_\ell(n,N)})
\]
where the $q_i$'s are certain type of bivariate polynomials,  $F=F(x_1,...,x_\ell)$ is a locally H\"older continuous function and the sequence of random variables $\{\xi_n\}$ satisfies some mixing and moment conditions. This paper continues the line of research started in \cite{Kifer2010} and \cite{KV}, and it is a generalization of the results in \cite{HK3} and Chapter 3 of \cite{book}. We will also prove a strong law of large numbers (SLLN) for the averages $N^{-\frac12}\cS_N(1)$ which extends the results from the beginning of Chapter 3 of  \cite{book} to general bivariate polynomial functions $q_i$.
Our results hold true for sequences $\{\xi_n\}$  generated by a wide class of Markov chains and dynamical systems. As an application we obtain  
 functional CLT's for expressions of the form $N^{-\frac12}M([Nt])$, where $M(N)$ counts the number of multiple recurrence of the sequence $\{\xi_n\}$ to certain sets $A_1,...,A_\ell$ which occur at the times $q_1(n,N),...,q_\ell(n,N)$, as well as SLLN's for these $M(N)$'s. One of the simplest examples is when $\xi_n$ is $n$-the digit of a random $m$-base or continued fraction expansion, and each $A_i$ is singleton (i.e. it represent one possible value of a digit).
\end{abstract}

\section{Introduction}
Since the ergodic theory proof of Szemer\'edi's theorem due to Furstenberg (see \cite{Fur1977}), ergodic theorems for  ``nonconventional" averages
\[
\frac 1N\sum_{n=1}^NT^{q_1(n)}f_1\cdots T^{q_\ell(n)}f_\ell
\]
became a well established field of research (the term "nonconventional" comes from \cite{Fu}). Here $T$ is a measure preserving
transformation, $f_1,...,f_\ell$ are bounded measurable functions and the $q_i$'s are
functions taking  positive integer values on the set of positive integers.
General polynomial $q_i$'s in this setup were first
considered in \cite{Be}. Taking $f_i$'s to be indicators of measurable sets
asymptotic results on numbers of multiple recurrences follow, which was
the original motivation for this study.

From a probabilistic point of view, ergodic theorems are laws of large numbers, and so  the question about other probabilistic limit theorems is natural.  In \cite{Kifer2010} and \cite{KV} central  limit theorems for random functions of the form
\begin{equation}\label{Sums}
N^{-\frac12}\sum_{n=1}^{[Nt]} F(\xi_{q_1(n)},\xi_{q_2(n)},...,\xi_{q_\ell(n)})
\end{equation}
were obtained. Here $F=F(x_1,...,x_\ell)$ is a locally H\"older continuous function, $\{\xi_n\}$ is a sequence of random variables satisfying some mixing and moment conditions and the $q_i$'s are functions satisfying certain growth conditions, which take positive integer values on the set of positive integers. Consedering polynomial $q_i$'s, the growth conditions \cite{Kifer2010} and \cite{KV} exclude the case when  some of the non-linear polynomials among $q_1,...,q_\ell$ have the same degree.
In \cite{HK3} we extended the above functional CLT's to the case when all of the $q_i$'s are polynomials, with no restrictions on their degrees. 

In this paper we will consider more general random functions of the form
\[
\cS_N(t)=N^{-\frac12}\sum_{n=1}^{[Nt]} F(\xi_{q_1(n,N)},\xi_{q_2(n,N)},...,\xi_{q_\ell(n,N)})
\]
where each $q_i$ is a bivariate polynomial with integer coefficients. The motivation for considering this case where the $q_i$'s depend also on $N$  comes from applications to multiple recurrence (see the last two paragraphs of this section), and  a relatively simple but still very interesting combinatorial number theoretic example can be described as follows. 
For each point $\om\in[0,1)$ consider its
base $m$ or continued fraction expansions with digits $\xi_k(\om)$, $k=1,2,...$. 
Next, count the number $S_N(\om)$ of those $\ell$-tuples 
$q_1(n,N),...,q_\ell(n,N),\,n\leq N$ for which, say, $\xi_{q_j(n,N)}(\om)=a_j,\,j=1,...,\ell$
for some fixed integers $a_1,...,a_\ell$. We have 
\[
S_N(\om)=\sum_{n=1}^N\prod_{j=1}^\ell\del_{a_j\xi_{q_j(n,N)}(\om)}
\]
where $\del_{ks}=1$ if $k=s$ and $=0$ otherwise. To make $\xi_k$'s random variable, we supply the unit interval with an 
appropriate probability measure such as the Lebesgue measure for base $m$ expansions
and the Gauss measure for continued fraction expansions (see Section \ref{Applications}). Now we can view $S_N$ as a random variable and 
set $\cS_N(t)=S_{[Nt]}$. 


In \cite{Kifer DCDS} several $L^2$ ergodic theorems were proved for the averages $N^{-\frac12}\cS_N(1)$, when $F(x_1,...,x_\ell)$ has the form $F(x_1,...,x_\ell)=\prod_{i=1}^\ell f_i(x_i)$. In \cite{Kifer SLLN} a strong law of large numbers (SLLN) was proved for $N^{-\frac12}\cS_N(1)$ when the $q_i$'s depend only on $n$, while in Chapter 3 of \cite{book}, under certain mixing conditions, we proved an SLLN  and
a CLT for $\cS_N(1)$  for linear functions $q_i(n,N)=a_in+b_iN$. The above results from \cite{Kifer SLLN} and \cite{book}
hold true for functions $F$ which do not necessarily have the above product form.
In this paper we will prove an SLLN for genera polynomials $q_i(n,N)$ and a functional CLT for certain classes of  polynomials $q_i(n,N)$, and our results generalize both \cite{HK3} and the above results from  Chapter 3 in \cite{book}.

A crucial step in proving a functional central limit theorem (regardless of the proofs method) is to show that the asymptotic covariances 
\[
b(t,s)=\lim_{N\to\infty}\bbE[\cS_N(t)\cS_N(s)]
\]
exist. Using some mixing conditions, we will show that these limits exist for several classes of bivariate polynomials $q_i(n,N)$. One of  the main difficulties arising here is to understand the asymptotic behaviour as $N\to\infty$ of certain sequences of sets $A_{N}\subset[1,N]$ which are related to  approximation properties of bivariate polynomial differences of the form 
$|q(m,N)-p(n,N)|$. This type of behaviour is investigated independently in Section \ref{Sec2}. In section \ref{PosSec} we address the question of the positivity of $D^2=\lim_{N\to\infty}\text{Var}(\cS_N(1))$, which is important since  $\cS_N$ converges towards the  process which equals identically $0$ when $D^2=0$. When any two non-linear polynomials $q_i$ and $q_j$ do not satisfy that $q_i(n,N)-q_j(cn+r,N)=z$, for some rational $c,r,z$, $c\not=1$, and all $n$ and $N$, or when the polynomials $q_i(n,N)$ are ordered so that $q_1(n,N)\leq q_2(n,N)\leq...\leq q_\ell(n,N)$ for any sufficiently large $n$ and $N$ then we give a complete characterization of the positivity of $D^2$. In more general circumstances we provide sufficient conditions for this positivity.

The CLT's from \cite{HK3} and Chapter 3 of \cite{book} rely on classical martingale approximation techniques, which was shown in \cite{KV} to be effective in the nonconventional setup. In the past decades Stein's method has become one of the main tools to prove central limit theorems. In \cite{Hafouta ECP} this method was applied successfully for nonconventional sums of the form (\ref{Sums}), and in Chapter 1 of \cite{book} we generalized \cite{Hafouta ECP}, and, in particular, showed that Stein's method yields a functional CLT in the case when the $q_i$'s depend only on $n$. Our proof of the CLT for the random functions $\cS_N$ defined above will also rely on an appropriate functional version  of Stein's method. When the appropriate limiting covariances exist, Stein's method in the functional setup can be applied successfully for random functions of the form $\sum_{n=1}^{[Nt]}X_{n,N}$ when the triangular array $\{X_{n,N}:\,1\leq n\leq N\}$ satisfies certain type of strong local dependence conditions, and our arguments will be based on showing that the summands in $\cS_N$ have such a local dependence structure. In fact, we will use this structure also to control the growth rate of the first four moments of $\cS_N(1)$, which is the key to the proof of the SLLN for $N^{-\frac12}\cS_N(1)$.

Our results hold true when, for instance,
$\xi_n=T^nf$ where $f=(f_1,...,f_d)$, $T$ is a topologically mixing subshift
of finite type, a hyperbolic diffeomorphism or an expanding transformation
taken with a Gibbs invariant measure, as well as in the case when
$\xi_n=f(\Upsilon_n), f=(f_1,...,f_d)$ where $\Upsilon_n$
is a Markov chain satisfying the Doeblin condition considered as a
stationary process with respect to its invariant measure. In the dynamical
systems case each $f_i$ should be either H\" older
continuous or piecewise constant on elements of Markov partitions. 
We can also extend our results to certain classes of dynamical systems $T$ which can be modelled by a Young tower (even though
the conditions specified in Section \ref{Sec1} do not seem to hold true, see Section  \ref{YoungSec}). As a consequence, our results hold true for
a variety of non-uniformly hyperbolic or distance expanding dynamical systems $T$, as well.
We refer the readers to Section \ref{Applications} for a detailed description of the sequences $\{\xi_n\}$  mentioned above.

As an application we can consider $F(x_1,...,x_\ell)=x_1^{(1)}\cdots
 x_\ell^{(\ell)}$,
$x_j=(x_j^{(1)},...,x_j^{(\ell)})$, $\xi_n=(X_1(n),...,X_\ell(n))$,
$X_j(n)=\bbI_{A_j}(T^n x)$ in the dynamical systems case and
$X_j(n)=\bbI_{A_j}(\Upsilon_n)$ in the Markov chain case where
$\bbI_{A}$ is the indicator of a set $A$. Let $M(N)$ be the number of $l$'s
between $0$ and $N$ for which $T^{q_{j}(l,N)}x\in A_j$ for $j=0,1,...,\ell$
(or $\Upsilon_{q_{j}(l,N)}\in A_j$ in the Markov chains case), where we set $q_{0}=0$, namely
the number of $\ell-$tuples of return times to
$A_j$'s (either by $T^{q_j(l,N)}$ or by $\Upsilon_{q_j(l,N)}$). Then
our results yield a functional central limit theorem for the number $M([Nt])$ and also an SLLN for $M(N)$.

\begin{acknowledgement*}
I would like to thank professor Yuri Kifer for suggesting me the problem and for several helpful discussions, as well.
\end{acknowledgement*}

\section{Limit theorems for nonconventional polynomial arrays}\label{Sec1}
\subsection{Preliminaries}
Our setup consists of a $\wp$-dimensional stochastic process $\{\xi_n,n\geq0\}$
on a probability space $(\Omega,\cF,P)$ and a family
of sub-$\sig$-algebras $\cF_{k,l}$, $-\infty\leq k\leq l\leq\infty$
such that $\cF_{k,l}\subset\cF_{k',l'}\subset\cF$ if $k'\leq k$
and $l'\geq l$. We will impose restrictions on the mixing coefficients
\begin{equation}\label{MixCoef1}
\phi(n)=\sup\{\phi(\cF_{-\infty,k},\cF_{k+n,\infty}): k\in\bbZ\}
\end{equation}
where for any two sub-$\sigma$-algebras $\cG$ and $\cH$ of $\cF$
\begin{equation}\label{general phi mixing}
\phi(\cG,\cH)=\sup\Big\{\Big|\frac{P(\Gam\cap\Del)}{P(\Gam)}-P(\Del)\Big|
: \Gam\in\cG,\,\Del\in\cH,\,P(\Gam)>0\Big\}.
\end{equation} 
The quantity $\phi(\cG,\cH)$ measures the dependence between the $\sigma$-algebras $\cG$ and $\cH$; it vanishes if and only if $\cG$ and $\cH$ are independent.

In order to ensure some applications,
in particular, to dynamical systems we will not assume that $\xi_n$
is measurable with respect to $\cF_{n,n}$ but instead impose restrictions
on the approximation rate
\begin{equation}\label{AprxCoef}
\beta_{q}(r)=\sup_{k\geq0}\|\xi_k-\bbE[\xi_k|\cF_{k-r,k+r}]\|_{L^q}.
\end{equation}
Our results will be obtained under the assumption that $\beta_{q}(n)$ and $\phi(n)$ converge sufficiently fast to $0$ (see (\ref{Mix0}) and Assumption \ref{Mix2} below), for a certain choice of $q$'s.
We note that
when the sequence $\{\xi_n,n\geq0\}$ itself satisfies some strong quantitative mixing conditions (e.g. when it forms a geometrcially ergodic Markov chain or when $\xi_n$ depends only on the $n$-th coordinate of a topologically mixing subshift of finite type, see Section \ref{Applications}), then our results  hold true when $\cF_{k,l}$ is the $\sigma$-algebra generated by $\xi_{\max(0,k)},...,\xi_{\max(0,l)}$. In this case we have $\beta_{q}(r)=0$ for any $q$ and $r$. 
When $\xi_n=f(T^n x)$ for several types of dynamical systems $T$ and functions $f$, the sequence $\{\xi_n,n\geq0\}$ will only be mixing in an ergodic theoretic sense, but when the function $f$ is H\"older continuous we will have that $\beta_{\infty}(r)\leq c\del^r$ for some $c>0$ and $\del\in(0,1)$ (see Section \ref{Applications}), where $\beta_{\infty}(r)$ is the approximation rate corresponding to some families of $\sigma$-algebras $\cF_{k,l}$ for which $\phi(n)\leq c\del^n$. 	For instance, when $T$ is a topologically mixing subshift of finite type then we can take $\cF_{k,l}$ to be the $\sigma$-algebra generated by the projections of the coordinates indexed by $\max(k,0),....,\max(l,0)$, while in the case when $T$ is some hyperbolic map we can take $\cF_{k,l}=\bigwedge_{j=k}^{l}T^{-j}\cM$, where $\cM$ is a Markov partition with sufficiently small diameter (see Section \ref{Applications}).


Next, we will discuss our stationarity  assumptions. 
We do not require the usual stationarity of the process $\{\xi_n, n\geq 0\}$, and instead we 
only assume that the distribution of $\xi_n$ does not depend on $n$ and that
the joint distribution of $(\xi_n,\xi_m)$ depends
only on $n-m$, which we write for further
reference by
\begin{equation}\label{xi dist}
\xi_n\stackrel {d}{\sim}\mu\,\,\,\text{ and }\,\, \big(\xi_n,\xi_m\big)\stackrel {d}{\sim}
\mu_{m-n}
\end{equation}
where $Y\stackrel {d}{\sim}\mu$ means that $Y$ has $\mu$ for its
distribution. 

Let $F=F(x_1,...,x_\ell)$, $x_j\in\bbR^\wp$
be a function on $(\bbR^\wp)^\ell$ which satisfies the following smoothness and growth conditions: 
for some $K>0$, $\iota\geq 0$,
$\kappa\in(0,1]$ and all $x_i,z_i\in\bbR^\wp$, $i=1,...,\ell$,
we have
\begin{equation}\label{F Hold}
|F(x)-F(z)|\leq K[1+\sum_{i=1}^\ell(|x_i|^\iota+|z_i|^\iota)]
\sum_{i=1}^\ell|x_j-z_j|^\kappa
\end{equation}
and
\begin{equation}\label{F Bound}
|F(x)|\leq K[1+\sum_{i=1}^\ell|x_i|^\iota]
\end{equation}
where $x=(x_1,...,x_\ell)$ and  $z=(z_1,...,z_\ell)$. We note that the role of (\ref{F Hold}) is to insure that we can approximate  $F(\xi_{n_1},\xi_{n_2},...,\xi_{n_\ell})$ by $F(\xi_{n_1,r},\xi_{n_2,r},...,\xi_{n_\ell,r})$ where $n_1,...,n_\ell $ are arbitrary  nonnegative integers,  $\xi_{n,r}=\mathbb E[\xi_n |\mathcal F_{n-r,n+r}]$ and $r$ is sufficiently large, while the role of (\ref{F Bound}) is to insure that $F(\xi_{n_1},\xi_{n_2},...,\xi_{n_\ell})$ will satisfy  some moment conditions (see a discussion after Assumption \ref{Moment Ass}).
Such an approximations is possible, of course, only when $\lim_{r\to\infty}\beta_q(r)=0$ (for some $q$) and the $\xi_i$'s satisfy certain moment conditions (such conditions are imposed in Assumptions \ref{Moment Ass} and \ref{Mix2} below).
We also remark that in the case when $\xi_n$ is measurable with respect to $\mathcal F_{n,n}$ then our results will follow with any
Borel function $F$ satisfying (\ref{F Bound}) without imposing (\ref{F Hold}), since then our proof do not require any approximations of the above form.

Our  moment conditions on the $\xi_n$'s are summarized in the following
\begin{assumption}\label{Moment Ass}
There exist $w>4$, $q\geq 1$ and $v>0$ such that 
\begin{equation}
\frac1w>\frac\iota v +\frac{\ka}q,\,\,\,\|\xi_n\|_{L^w}<\infty,\,\,\text{ and }\,\|\xi_n\|_{L^q}<\infty
\end{equation}
where $\iota$ and $\ka$ come from (\ref{F Hold}) and (\ref{F Bound}).
\end{assumption}
We note that Assumption \ref{Moment Ass}
 insures that $\|F(\xi_{n_1},\xi_{n_2},...,\xi_{n_\ell})\|_{L^w}\leq C$ for some $C>0$ and all nonnegative integers $n_1,...,n_\ell$. When $\phi(n)$ and $\beta_{q_0}^\kappa(n)$ (for some $q_0$) converge exponentially fast to $0$ as $n\to\infty$ then in order to get the strong law of large numbers (Theorem \ref{LLN}) and the functional central limit theorem (Theorem \ref{CLTthm}) we can take any $w>4$, $q=q_0$ and any $v>0$ so that $\frac1w>\frac\iota v +\frac{\ka}{q_0}$. We also remark that we have presented the moment conditions with some dependence on the regularity on $F$ (i.e. on $\ka$ and $\iota$), but, of course, we could have first assumed that the properties described in Assumption \ref{Moment Ass} hold true with some $w,q,v,\ka\in(0,1]$ and $\iota\geq0$, and then consider only functions satisfying (\ref{F Hold}) and (\ref{F Bound}) with these $\ka$ and $\iota$.

Next, let $q_1(n,N),...,q_\ell(n,N)$ be polynomials with nonnegative integer coefficients which do not depend only on $N$.
We assume here, for the sake of convenience, that $\deg q_i\leq\deg q_{i+1}$ for any $i=1,2,...,\ell-1$, where the degree of a bivariate polynomial $p(x,y)$ is the degree of the univariate polynomial $p(x,x)$, and that 
 the differences $q_i-q_j$ are not constants (the case of constant difference can be treated as in Section 3 of \cite{HK3}).
For each $N$ set 
\[
S_N=\sum_{n=1}^{N} F(\xi_{q_1(n,N)},\xi_{q_2(n,N)},...,\xi_{q_\ell(n,N)}).
\]

The first result we will prove in this paper is the following strong law of large numbers:
\begin{theorem}\label{LLN}
Suppose that Assumption \ref{Moment Ass} holds true with numbers $w$ and $q$ so that
\begin{equation}\label{Mix0}
\sum_{l=1}^\infty l\big((\phi(l))^{1-\frac 4w}+\beta_q^\ka(l)\big)<\infty
\end{equation}
where $\ka$ comes from \ref{F Hold}.
Then, $P$-almost surely we have 
\[
\lim_{N\to\infty}\frac1N S_N=\bar F:=\int F(x_1,...,x_\ell)d\mu(x_1)\dots d\mu(x_\ell).
\]
\end{theorem}
The reason that the mixing and approximation coefficients $\phi(l)$ and $\beta_q(l)$ appear additively in (\ref{Mix0}) is that in the proof of Theorem \ref{LLN} we approximate each summand  $F(\xi_{q_1(n,N)},\xi_{q_2(n,N)},...,\xi_{q_\ell(n,N)})$ by 
$F(\xi_{q_1(n,N),r(n)},\xi_{q_2(n,N),r(n)},...,\xi_{q_\ell(n,N),r(n)})$, where $r(n)$ is some number which depends on $n$ and $\xi_{m,r}=\mathbb E[\xi_m|\cF_{m-r,m+r}]$ for any $m$ and $r\geq0$ (or course, we will require that $\lim_{n\to\infty}r(n)=\infty$ at a certain rate).

The main result in this section is a functional central limit theorem for the sequence of random functions $\cS_N:[0,1]\to\bbR,\,N\in\bbN$ given by
\[
\cS_N(t)=N^{-\frac12}\sum_{n=1}^{[Nt]} F(\xi_{q_1(n,N)},\xi_{q_2(n,N)},...,\xi_{q_\ell(n,N)})=N^{-\frac12}S_{[Nt]}.
\]
To simplify  formulas we assume the (asymptotic) centering condition 
\begin{equation}\label{F bar}
\brF =\int F(x_1,...,x_\ell)d\mu(x_1)\dots d\mu(x_\ell)=0
\end{equation}
which is not really a restriction since
we can always replace $F$ by $F-\brF$. We note that $\brF$ is not the expectation of $F(\xi_{q_1(n,N)},\xi_{q_2(n,N)},...,\xi_{q_\ell(n,N)})$, but the arguments in our proofs show that, in the circumstances of Theorem \ref{LLN}, the expectation of $\frac 1 N S_N$ converges to $\bar F$ as $N\to\infty$ (this is what  expected to happen, in view of Theorem \ref{LLN}).

Our main mixing and approximation conditions for the central limit theorem is the following 
\begin{assumption}\label{Mix2}
There exist $d\geq 1$ and $\te>2$ such that for any $n\in\bbN$,
\begin{equation}\label{Mix2cond}
\phi(n)+\big(\beta_{q}(n)\big)^\ka\leq dn^{-\te}.
\end{equation}
\end{assumption}
In contrast to the proof of Theorem \ref{LLN}, in the proof of the CLT we will approximate each summand  $F(\xi_{q_1(n,N)},\xi_{q_2(n,N)},...,\xi_{q_\ell(n,N)})$ by 
$F(\xi_{q_1(n,N),r(N)},\xi_{q_2(n,N),r(N)},...,\xi_{q_\ell(n,N),r(N)})$, where $r(N)$ depends on $N$. This is the reason that $\phi(n)$ and $\beta_{q}(n)$ do not appear additively in (\ref{Mix2cond}).

\subsection{Classes of polynomials}
We  describe here several classes of polynomials for which we can derive the weak invariance principle for the random functions $\cS_N(\cdot)$. 

First, we  assume here that the linear polynomials among the $q_i$'s have the form 
\begin{equation}\label{linear}
q_i(n,N)=a_in+b_iN
\end{equation}
for some integers $a_i$ and $b_i$, namely that $q_i(0,0)=0$. Our additional requirements from the linear polynomials are described in the following
\begin{assumption}\label{A1}
For any linear $q_i$ and $q_j$
the difference $a_i-a_j$ is
divisible by the greatest common divisor of $b_i$ and $b_j$ where the
$a_i$'s and $b_i$'s are the same as in (\ref{linear}). 
\end{assumption}

Next, in order to describe our conditions regarding the non-linear polynomials among the $q_i$'s, we need the following definitions.
Let $q(n,N)$ and $p(n,N)$ be two bivariate polynomials with nonnegative integer coefficients. We will say that $q$ and $p$ have \textit{exploding differences }
if  for any $\del\in(0,1)$ there exist constants $C_\del>0$ and $N_\del$ and sets $\Gam_{N,\del}\subset[1,N]$, whose cardinality does not exceed $\del N$, so that for any $N>N_\del$ and $n\in[\del N,N]\setminus\Gam_{N,\del}$,
\begin{equation}\label{Exp dif}
\min_{m\in[\del N,N]}|q(m,N)-p(n,N)|\geq C_\del N.
\end{equation}
It is clear that any two polynomials $q$ and $p$ with different degrees have exploding differences and that two linear polynomials do not have exploding differences.
In Section \ref{Sec2} we will give several classes of examples of  polynomials $q$ and $p$ with the same non-linear degree which have exploding differences. 




Next,
for any $1\leq i\leq\ell$ such that $\deg q_i=k>1$ write 
\begin{eqnarray}
q_i(n,N)=\sum_{u=0}^{k}N^u Q_{i,u}(y)
\end{eqnarray}
where $y=n/N$ and each $Q_{i,u}$ is a polynomial with non-negative integer coefficients whose degree does not exceed $u$.
For any distinct $1\leq i,j\leq\ell$ such that $\deg q_i=\deg q_j=k>1$, we will say that $q_i$ and $q_j$ are \textit{linearly related} if  $Q_{i,k}$ and $Q_{j,k}$ are not constants and there exist constants $c_{i,j},r_{i,j}\in\bbR$, $c_{i,j}>0$ so that $Q_{j,k}(c_{i,j}y)=Q_{i,k}(y)$
and 
\[
Q_{i,k-1}(y)-Q_{j,k-1}(c_{i,j}y)=r_{i,j}Q_{j,k}'(c_{i,j}y)
\]
for any $y\in[0,1]$. Then, any two polynomials $q_i$ and $q_j$ which do not depend on $N$ and have the same non-linear degree $k$ are linearly related. Indeed, in this case we have $Q_{i,u}(y)=a_{i,u}y^u$ and $Q_{j,u}(y)=a_{j,u}y^u$ for some integers $a_{i,u}$  and $a_{j,u}$ so that $a_{i,k},a_{j,k}>0$, and so we can take 
$
c_{i,j}=(\frac{a_{i,k}}{a_{j,k}})^{1/k}$ and, with $c_{i,j,k}=(c_{i,j})^{k-1}$, 
$r_{i,j}=\frac{a_{i,k-1}-a_{j,k-1}c_{i,j,k}}{ka_{j,k}c_{i,j,k}}.
$ 
This means that all the results obtained in this paper generalize the results from \cite{HK3}, in which a nonconventional polynomial CLT was obtained in the case when all the $q_i$'s are polynomial functions of the variable $n$. Observe also that the linear relation condition  involves only the $Q_{i,k}$'s and $Q_{i,k-1}$'s and
note that $q_i$ and $q_j$ (with the same non-linear degree $k$) are linearly related if 
\begin{eqnarray*}
q_i(n,N)=\alpha_in^{s}N^{k-s}+\beta_in^{s-1}N^{k-s}+G_i(n,N)\\\text{ and }\,\, 
q_j(n,N)=\alpha_jn^{s}N^{k-s}+\beta_jn^{s-1}N^{k-s}+G_j(n,N)
\end{eqnarray*}
for some $0<s\leq k$, polynomials  $G_i$ and $G_j$ whose degree does not exceed $k-2$ and positive integers $\al_i,\al_j,\be_i$ and $\be_j$.
We refer the readers to Corollary \ref{Cor linear case} for a characterization when two linearly related polynomials have exploding differences (see also Remark \ref{Remark equiv} below).

We will obtain out results under the following 

\begin{assumption}\label{A2}
Any two non-linear polynomials $q_i$ and $q_j$  are either linearly related, or the differences of $q_i$ and $q_j$ explode.  
\end{assumption}

\subsection{Limiting covariances and the CLT}
Our first CLT related result is the following
\begin{theorem}\label{VarThm}
Suppose that  Assumptions \ref{Moment Ass} and \ref{Mix2} are satisfied with
numbers $w$ and $\te$ so that  $\te>\frac{4w}{w-2}$.
Assume, in addition, that  Assumptions \ref{A1} and \ref{A2} are satisfied.
 Then the limits 
\[
b(t,s)=\lim_{N\to\infty}\bbE [\cS_N(t)\cS_N(s)]
\]
exist, where $0\leq t,s\leq 1$. In particular, the limit 
\[
D^2=\lim_{N\to\infty}\bbE S_N^2=\lim_{N\to\infty}\text{Var}(S_N)
\]
exists, where $S_N=\cS_N(1)$.
\end{theorem}
Note that in Section \ref{Var Sec} we will also provide several formulas for the limits $b(t,s)$, as well as a some conditions for the positivity of $D^2$ (when $D^2=0$ then $\cS_N$ converges to the  process which equals $0$ identically). We refer the readers to Section \ref{General Cov} to a discussion about existence of $b(t,s)$ (or just $D^2$) for (more general) polynomials $q_i$'s satisfying  certain number theory related conditions.

\begin{remark}\label{Remark equiv}
The property of being linearly related is, in fact, an equivalence relation. Indeed, if both pairs $(q_i,q_j)$ 
and $(q_j,q_l)$ are linearly related then we can always take $c_{i,l}=c_{i,j}\cdot c_{j,l}$ and 
\[
r_{i,l}=r_{j,l}+r_{i,j}\cdot c_{j,l}.
\]
We will say that the polynomials $q_i$ and $q_j$ are $\bbQ$-equivalent if there exist rational $c$ and $r$ so that the 
difference $q_i(n,N)-q_{j}(cn+r,N)$ does not depend on $n$ and $N$. Then any two $\bbQ$-equivalent polynomials are linearly related.
In Corollary \ref{Cor linear case} we will show that any two linearly related polynomials which are not $\bbQ$-equivalent have exploding difference. 
Therefore Assumption \ref{A2} means that any two non-linear  polynomials among the $q_i$'s are either equivalent or have exploding differences, and under Assumption \ref{A2} having exploding differences is a symmetric relation. 
\end{remark}


When the asymptotic covariances $b(t,s)$ exist then, using a functional version of Stein's method due to A.D. Barbour, we derive the following
\begin{theorem}\label{CLTthm}
Suppose that  Assumptions \ref{Moment Ass} and \ref{Mix2} are satisfied with
numbers $w$ and $\te$ so that $\te>\frac{4w}{w-2}$. Assume, in addition,
that the limiting covariances $b(t,s)$ exist. Then, the random functions $\cS_N:[0,1]\to\bbR$ converge in distribution as $N\to\infty$ towards a centered Gaussian process $\eta(t)$ whose covariances are given by 
$
\bbE [\eta(t)\eta(s)]=b(t,s).
$
\end{theorem}
The arguments in the proof of Theorem \ref{CLTthm} together with the arguments in Chapter 1 of \cite{book} show that Stein's method also yields almost optimal convergence rate in the CLT for the sequence of random variables $S_N=\cS_N(1)$, when $D^2>0$. These results are not included here in order not to overload this paper.

\subsection{Outline of the proof of Theorems \ref{LLN}, \ref{VarThm} and \ref{CLTthm}}
Consider first the case when the random variables $\xi_n$ are independent. For any $n$ and $N$ set 
\[
F_{n,N}=F(\xi_{q_1(n,N)},\xi_{q_2(n,N)},...,\xi_{q_\ell(n,N)}).
\]
Fix some $N$. Since the $q_i$'s are polynomials, each $F_{n,N}$ can depend on at most $d_0\ell^2$ of the random variables $F_{m,N},\,m\geq1$, where $d_0$ is the maximal degree of the polynomials $q_i$. We see then that the triangular array $\{F_{n,N},\,n=1,2,...,N\}$ satisfies a ``local dependence" type condition (here we view the indexes $m$ of the random variables $F_{m,N}$ which depend on $F_{n,N}$ as a certain ``neighborhood of dependence"). 
When the sequence $\{\xi_n\}$ is only weakly dependent then we get a certain version of the above local dependence, but now, roughly speaking, the dependence is replaced with a certain type of ``local strong dependence", which essentially means that each $F_{n,N}$ can ``strongly depend" only a number of $F_{m,N}$'s whose magnitude is smaller than $N$. Giving a precise meaning to this ``strong local dependence" is one of the main ideas behind the proof of Theorem \ref{CLTthm} (see the beginning of Section \ref{CLT sec} for the precise definition of the underlying graph in the weakly depend case).
This was done in Chapter 1 of \cite{book} in the case when $q_i$ depends only on $n$ (in fact, the full details were given there only when $q_i(n)=in$ for each $i$), but here the dependence on $N$ causes additional difficulties. 
The resulting local dependence structure is a classical situation where we can use the functional version of Stein's method for Gaussian approximation due to A.D. Barbour (see Theorem \ref{ThmFunc}  and \cite{Barb}).

Using the above strong local dependence structure together with some counting arguments, we get that the $L^4$ norms the sums $N^{-\frac12}(S_N-N\bar F)$ are bounded in $N$. This is the main idea of the proof of Theorem \ref{LLN}.  Indeed, assume for the sake of convenience that $\bar F=0$. Then by the Markov inequality we have
\[
P\big\{N^{-1}|S_N|\geq N^{-1/8}\big\}\leq\frac{\|S_N\|^4_{L^4}}{N^{7/2}}\leq C N^{-3/2}
\]
which together with the Borel-Cantelli lemma yields that with probability one we have
$
\lim_{N\to\infty}\frac1{N}S_N=0=\brF.
$

The proof of Theorem \ref{VarThm} is somehow less transparent even in the case of independent $\xi_n$'s (and it does not rely only on strong local dependence), but let us describe some of  its key ingredients. In the case when the $q_i$'s depend only on $n$, we can order them and just assume that $q_1(n)<...<q_\ell(n)$ for any sufficiently large $n$. When the polynomials $q_i$ depend on $n$ and $N$ then, in general, it is impossible to order them. The first step in the proof of Theorem \ref{VarThm} is to show that, after omitting a ``small" number of $n$'s (in comparison to $N$) between $1$ to $N$, we can order the polynomials $q_{i}(n,N)$ when $n\in\Gamma_i,\,i=1,2,...,d$, where $\{\Gamma_i\}$ is a partition of the remaining $n$'s. After this is established, roughly speaking, our proof scheme requires to study certain combinatorial number theory related problems which are related to asymptotic densities of sets of the form
\[
\big\{aN\leq n\leq bN:\,q(n,N)\in\cup_{m\in\bbN}\{p(m,N)\}\big\}\subset[1,N]
\]
where $a,b$ are positive numbers between $0$ to $1$ and $q$ and $p$ are bivariate polynomials with integer coefficients. 
Such sets were invesigated in \cite{HK3} when the polynomials $q(x,y)$ and $p(x,y)$ depend only on $x$, but when they depend also on $y$ many additional number theoretic difficulties arise.


\section{Differences of bivariate non-linear polynomials}\label{Sec2}

Let $q(n,N)=q_N(n)$ and $P(n,N)=P_N(n)$ be two polynomials in the variables $n,N$ with nonnegative integer coefficients so that $\deg q=\deg p=k>1$ for some $k>1$. We will also assume here that the polynomials $q$ and $p$ do not depend only $N$. In particular the functions $q_N^{-1}:[q_N(0),\infty)\to[0,\infty)$ and $p_N^{-1}:[p_N(0),\infty)\to[0,\infty)$ are well defined.
The goal in this section is to investigate the asymptotic behaviour of the differences $|q_N(m)-p_N(n)|$. In Sections \ref{First est} and \ref{Sec hom decomp} we will prove some general  results, which will be applied in Sections \ref{LR sec} and \ref{frac sec} in more specific situations.


\subsection{First estimate}\label{First est}
Our first results is the following

\begin{proposition}\label{Prop1}
Suppose that $q$ and $p$ have the form 
\begin{equation}\label{N rep}
p(n,N)=H(N)P(n,N)+r(n,N)\,\,\text{ and }q(m,N)=H(N)Q(m,N)+s(m,N)
\end{equation}
for some non-constant polynomial $H$, polynomials $Q$ and $P$ with non-negative integer coefficients and  polynomials $r$ and $s$ so that $\max(\deg s, \deg r)<\deg H$. Then for any $1\leq n,m\leq N$ so that  $p_N(n)>q_N(0)$ and $q_N(m)>p_N(0)$,
either $Q_N(m)=P_{N}(n)$, where $Q_N(x)=Q(x,N)$ and $P_N(x)=P(x,N)$, or 
\begin{equation}\label{Prop 1 main}
\big|q_N(m)-p_N(n)\big|\geq \Del_N(n,m)|H(N)|-|s(m,N)-r(n,N)|,
\end{equation}
where 
$\Del_N(n,m)$ is the minimum of $Q_N'(Q_N^{-1}P_N(n))/Q_N'(m)$ and $P_N'(P_N^{-1}Q_N(m))/P_N'(n)$
and $P_N^{-1}$, and $Q_N^{-1}$ are the inverse functions of the univariate functions $P_N(\cdot)$ and $Q_N(\cdot)$, respectively.
In particular, for any $\del>0$ there exists a constant $R_\del>0$ so that for any sufficiently large $N$ and $\del N\leq n,m\leq N$ so that  $Q_N(m)\not=P_{N}(n)$,
\[
\big|q_N(m)-p_N(n)\big|\geq R_\del|H(N)|.
\]
As a consequence, $q$ and $p$ have exploding differences when $Q$ and $P$ have exploding differences or when 
the degrees of $r$ and $s$ are different.
\end{proposition}
The polynomials $Q$ and $P$ have exploding differences when they are linearly related, but not $\bbQ$-equivalent (see Remark \ref{Remark equiv} and Corollary \ref{Cor linear case}). They also have exploding differences  in the circumstances of  Corollary \ref{N cor}. Set
\[
\bar d=\limsup_{N\to\infty}\frac1N\big|\big\{1\leq n\leq N:\,P(n,N)\in\{Q(m,N):\,m\in[1,N]\}\big\}\big|
\]
where $|\Gamma|$ stands for the cardinality of a finite set $\Gamma$. Then, it follows from Proposition \ref{Prop1} that the polynomials $q$ and $p$ have exploding differences also when $\bar d=0$. The upper limit $\bar d$ equals $0$ when $Q$ and $P$ have exploding differences, but also when, for instance, $P(n,N)$ and $Q(m,N)$ take values at disjoint sets (e.g. when $P(n,N)$ is odd and $Q(m,N)$ is even etc.).

\begin{proof}[Proof of Proposition \ref{Prop1}]
It is clearly enough to prove (\ref{Prop 1 main}) in the case when  $r\equiv 0$ and $s\equiv 0$. 
Since $p_N(n)>q_N(0)$ and $q_N(n)>p_N(0)$,
the numbers $t_{n,N}=q_N^{-1}p_N(n)=Q_N^{-1}P_N(n)$ and $s_{m,N}=p_N^{-1}q_N(m)=P_N^{-1}Q_N(m)$ are well defined. Suppose first that $q_N(m)>p_N(n)$. Then $Q_N(m)>P_N(n)$ and 
 we can write $m=t_{n,N}+x$, where here $x\geq0$ is considered as a parameter. 
Define the function $D_{n,N}(y)$ by 
\[
D_{n,N}(y)=Q_{N}(t_{n,N}+y)-P_N(n)=Q_{N}(t_{n,N}+y)-Q_N(t_{n,N}).
\]
Then $D_{n,N}(0)=0$.
Applying the mean value theorem with the function $D_{n,N}$, taking into account that the derivative of $Q_N$ is increasing and that $x\geq0$, we obtain that
\[
|Q_{N}(m)-P_N(n)|=|D_{n,N}(x)-D_{n,N}(0)|\geq Q_{N}'(t_{n,N})|\cdot|x-0|=Q_{N}'(t_{n,N})\cdot|m-t_{n,N}|.
\]
Next, we define the function $g=g_{N,n}(\cdot)$ by
$
g(t)=Q_N(t)-P_N(n).
$
Then $g(t_{n,N})=0$.
By the mean value theorem, there exists $\xi$ between $m$ and $t_{n,N}$ so that
\[
|m-t_{n,N}|\cdot g'(\xi)=|g(m)-g(t_{n,N})|=|g(m)|\geq 1
\]
where we used that that $m\not=t_{n,N}$. Since $0\leq g'(\xi)=Q_N'(\xi)\leq Q_N'(m)$  we
obtain that
\[
|m-t_{n,N}|\geq \frac{1}{Q'_N(m)},
\]
which together with the previous estimates implies that 
\[
|Q_{N}(m)-P_N(n)|\geq\frac{Q_N'(Q_N^{-1}P_N(n))}{Q_N'(m)}.
\]
In the case when $Q_N(m)<P_N(n)$ we obtain (\ref{Prop 1 main}) by
reversing the roles of $Q$ and $P$ and the above arguments.
\end{proof}

 We refer the readers to Corollary \ref{N cor}, in which we give a class of examples of polynomials $Q$ and $P$ with exploding differences. 
 
In the case when $r$ and $s$  are polynomials of the same degree,  we can check whether Proposition \ref{Prop1} can be applied with $s$ and $r$ in place of $q$ and $p$. Still, $r$ and $s$ (or even $p$ and $q$) may, for instance, contain a monomial which does not depend on $N$. In the next section we will estimate $|q_N(m)-p_N(n)|$ under somehow different type of conditions, which will have applications beyond the case considered in Proposition \ref{Prop1}.


\subsection{Estimates using decompositions into homogeneous polynomials}\label{Sec hom decomp}
Set $y_1=\frac{n}{N}$, $y_2=\frac{m}{N}$ and write
\begin{equation}
p_N(n)=\sum_{j=0}^{k}N^j P_j(y_1)\,\,\text{ and }\,\,q_N(m)=\sum_{j=0}^{k}N^j Q_j(y_2)
\end{equation}
where $P_j$ and $Q_j$ are polynomials whose degree does not exceed $j$.  
We will also assume here that $Q_k$ and $P_k$ are not constant polynomials and that $Q_k(0)\leq P_k(0)$.
In the above circumstances, the function $\gam_k(y)=Q_k^{-1}P_k(y)$ is well defined on $[0,\infty)$.

Next, for any $y\in(0,1]$, let the polynomial $H_{N,y}$ be given by
\begin{eqnarray*}
H_{N,y}(\xi)=\sum_{u=2}^{k-1}N^{-(u-1)}C_u(y)+\big(1+\sum_{u=2}^{k}N^{-(u-1)}A_{1,u}(y)\big)\xi\\+
\sum_{s=2}^k\big(\sum_{u=2}^k N^{-(u-1)}A_{s,u}(y)\big)\xi^s.
\end{eqnarray*}
Here $A_{s,u}(y)=\sum_{j=s}^{u}\binom{j}{s}\frac{Q_{k-j+u}^{(j)}(\gam_k(y))}{Q_k'(\gam_k(y))}(R_k(y))^{j-s}$
and
\begin{equation}\label{C u def}
C_u(y)=\frac{Q_{k-u}(\gam_k(y))-P_{k-u}(y)}{Q_k'(\gam_k(y))}
\end{equation}
where
$
R_k(y)=\big(P_{k-1}(y)-Q_{k-1}(\gam_k(y))\big)/Q_k'(\gam_k(y)).
$
Note that when $k=2$ then we set $\sum_{u=2}^{k-1}N^{-(u-1)}C_u(y)=0$.
Since the functions $Q_k'\circ \gam_k, C_u$ and $A_{s,u}$ are bounded on $[0,1]$, 
it is clear that there exists a constant $A_1$ which depend only on  the polynomials $q$ and $p$ so that for any $y\in(0,1]$,
\begin{equation}\label{C a}
\sup_{\xi\in [-1,1]} |H_{N,y}'(\xi)-1|\leq \frac{A_1}{NQ_k'(\gam_k(y))}.
\end{equation}
Therefore, there exist a constant $N_0$ so that if $Q_k'(\gam_k(y))> \frac{A_1}{N}$ and $N>N_0$ is sufficiently large then the function $H_{N,y}$ is strictly increasing on $[-1,1]$ and 
there exists a unique root $x_N(y)$ of $H_{N,y}$ in $[-1,1]$, which, by the mean value theorem, 
satisfies that $H_{N,y}(0)=A_N(y)x_N(y)$, for some function $A_N(y)$ so  that 
\begin{equation}\label{A function}
|A_N(y)-1|\leq\frac{A_1}{NQ_k'(\gam_k(y))}.
\end{equation}
Observe that $H_{N,y}(0)$ is at most of order $N^{-1}$. When all of the functions $C_2,...,C_{k-1}$ are identically $0$ then $H_{N,y}(0)=x_N(y)=0$. In general, we have the following

\begin{lemma}\label{lemma}
Suppose that not all the $C_u$'s are identically zero.
Let  $s_0\leq k-1$ denote the first index $u$ so that the function $C_u(\cdot)$ does not equal $0$ identically.
Then for any $\del\in(0,1)$ there exist positive constants $B_1(\del)$ and $B_2(\del)$ and a set $\Gam_{N,\del}\subset[1,N]$, whose cardinality does not exceed $\del N$, so that for any sufficiently large $N$ and $n\in[\del N,N]\setminus\Gam_{N,\del}$,
\[
B_1(\del) N^{-(s_0-1)}\leq |H_{N,n/N}(0)|\leq B_2(\del) N^{-(s_0-1)}.
\]
\end{lemma}

\begin{proof}
The function $C_{s_0}$ can have only a finite number of roots $y_1,...,y_t$ in the interval $[0,1]$ (since $C_{s_0}$ can be extended to an analytic function in a complex neighborhood of $[0,1]$). Let us denote these roots by $y_1,...,y_t$. Let $\del>0$. Then there exists a constant $C_\del>0$ so that for any $y\in[0,1]$ which satisfy that
\begin{equation}\label{y i dist}
\min_{1\leq i\leq t}|y-y_i|\geq\frac{\del}{2t}:=\del_t
\end{equation}
we have
$
|C_{s_0}(y)|\geq C_\del.
$
Let $B$ be an upper bound of the absolute value of the functions $C_u:[0,1]\to\bbR$. Set 
\[
\Gamma_{N,\del}=\bigcup_{i=1}^{t}[N(y_i-\del_t),N(y_i+\del_t)].
\]
If $y:=n/N\in[\del,1]\setminus\Gamma_{N,\del}$, then, with $s=s_0-1$
\begin{equation*}
(C_\del-BkN^{-1})N^{-s}\leq |H_{N,y}(0)|=\Big|\sum_{u=s_0}^{k}N^{-(u-1)}C_u(y)\Big|\\\leq kBN^{-s}
\end{equation*}
and the proof of the lemma is complete.
\end{proof}

Our next result is the following 

\begin{lemma}\label{Lemma}
For any natural $n,m$ and $N$, 
with $y=\frac{n}{N}$, we have
\[
q_N(m)-p_N(n)=q_0(R_k(y))-P_0+Q_k'(\gam_k(y))N^{k-1}H_{N,y}(m-N\gam_k(y)-R_k(y)).
\]
As a consequence, if $Q_k'(\gam_k(y))> \frac{2A_1}{N}$ and $N$ is sufficiently large, where $A_1$ comes from (\ref{C a}), then, with $\al_N(y)=N\gam_k(y)+R_k(y)+x_N(y)$,
\begin{eqnarray*}
\frac12Q_k'(\gam_k(y))N^{k-1}|m-\al_N(y)|\\\leq \left|q_N(m)-p_N(n)-\big(Q_0(R_k(y))-P_0\big)\right|\leq  2Q_k'(\gam_k(y))N^{k-1}|m-\al_N(y)|.
\end{eqnarray*}
In particular, for any constants $s<k$ and $0<B_1(y)<B_2(y)<\infty$ so that 
\[
B_1(y) N^{-s}\leq |H_{N,y}(0)|\leq B_2(y) N^{-s}
\] 
there exists a constant $K(y)$ so that
\[
|q_N(m)-p_N(n)|\geq K(y) N^{k-s}-|q_0(R_k(y))-P_0|
\]
if 
\begin{equation}\label{Lower bound}
\big|m-N\gam_k\big(\frac{n}{N}\big)-R_k\big(\frac{n}{N}\big)\big|\geq KN^{-(s-1)}
\end{equation}
for some $K>0$. The constant $K(y)$ depends only on $B_1(y),B_2(y),K$ and $s$. 
\end{lemma}
Let $\del\in(0,1)$. When $y=n/N\in[\del,1]$ then $Q_k'(\gam_k(y))\geq C_\del$ for some $C_\del>0$ which depends only on $\del$, and so the magnitude of  $Q_k'(\gam_k(y))N^{k-1}$ is $N^{k-1}$ and the inequality  $Q_k'(\gam_k(y))> \frac{2A_1}{N}$ holds true, assuming that $N$ is sufficiently large. Observe also that for such $n$'s we have 
\[
\max_{y\in[C_\del,1]}|q_0(R_k(y))-P_0|=D_\del<\infty
\]
and that $D_\del$ depends only on $\del$ (and on the polynomials $q$ and $p$). Therefore, when $n\in[\del N,N]\setminus\Gam_{N,\del}$, where $\Gam_{N,\del}$ comes from Lemma \ref{lemma} with $s_0=s+1$, then 
\[
|q_N(m)-p_N(n)|\geq K_\del N^{k-s}
\] 
for any $n$ and $m$  satisfying (\ref{Lower bound}), for some constant $K_\del>0$ (and so, the problem of verifying that $q$ and $p$ have exploding differences is reduced to the study of (\ref{Lower bound})-see Corollary \ref{fractional corollary} for an application).

\begin{proof}[Proof of Lemma \ref{Lemma}]
Write $y=\frac{n}N$ and $m=N\gam_k(y)+x$, where $x$ is considered here as a parameter. Then, by considering the Taylor expansion of the polynomials $Q_j$ around the point $\gam_k(y)$ we have
\begin{eqnarray*}
q_N(m)-p_N(n)=\sum_{j=0}^{k}N^j\Big(\sum_{s=0}^j\frac{Q_j^{(s)}(\gam_k(y))}{s!}\big(\frac{x}{N}\big)^s-P_j(y)\Big)\\=\sum_{u=0}^{k}N^{k-u}\Big(\sum_{s=0}^u\frac{Q_{s+k-u}^{(s)}(\gam_k(y))x^s}{s!}-P_{k-u}(y)\Big)
\\=N^{k-1}Q_k'(\gam_k(y))(x-R_k(y))+\sum_{u=2}^{k}N^{k-u}\Big(\sum_{s=0}^u\frac{Q_{s+k-u}^{(s)}(\gam_k(y))x^s}{s!}-P_{k-u}(y)\Big)
\end{eqnarray*}
where we used that $Q_k(\gam_k(y))=P_k(y)$.
By considering the above expression as a (polynomial) function of $x$ (where $y$ is considered as a parameter), and then considering its Taylor polynomials around  the point $R_k(y)$ we arrive at
\begin{eqnarray*}
q_N(m)-p_N(n)=Q_k'(\gam_k(y))N^{k-1}H_{N,y}(x-R_{k}(y))+\sum_{s=0}^{k}q_{s,s}\big(R_k(y)\big)^{s}-P_0\\=
Q_k'(\gam_k(y))N^{k-1}H_{N,y}(x-R_{k}(y))+q_0(R_k(y))-P_0
\end{eqnarray*}
where $H_N$ was defined in the statement of the lemma  $q_{s,s}$ is the coefficient of monomial $y^s$ in the polynomial $Q_s(y)$.
\end{proof}

In the following sections we will apply Lemma \ref{Lemma} in several situations, where $\gam_k$ and $R_k$ are assumed to have certain  structure. In Section \ref{Sec Var proof} we will use the results from Sections \ref{LR sec} and \ref{frac sec} in order to prove Theorem \ref{VarThm}. For some abstract application of Lemma \ref{Lemma}, we refer the readers to Section \ref{General Cov}.

\subsection{Application I: linearly related polynomials}\label{LR sec}
We begin with the following consequence of Lemma \ref{Lemma}. 

\begin{corollary}\label{Cor linear case}
Suppose  that $\gam_k(y)=cy$, $R_k(y)=r$ for some constants $c$ and $r$, namely that  $q$ and $p$ are linearly related.
Then either $q$ and $p$ have exploding differences, or $q$ and $p$ are $\bbQ$-equivalent (in the terminology of Remark \ref{Remark equiv})  and then for any $\del\in(0,1)$ there exist constants $W_\del$  and $N_\del$ so that for any $N>N_\del$,  $\del N\leq n\leq N$ and $m\in\bbN$ either $m=cn+r$  (which happens on a finite union of arithmetic progressions) and $q_N(m)-p_N(n)=q_0(r)-P_0:=d$ or $|q_N(m)-p_N(n)|\geq W_\del N^{k-1}$.
\end{corollary}
\begin{proof}
For each $n$ and $N$ set $d_N(n)=\inf\{|q_N(m)-p_N(n)|:\, m\in\bbN\}$.
 Let $\del>0$. Suppose first that $c$ is irrational. Then, using Weyl's equidistribution theorem, we see that
\[
\limsup\frac1{N}\left|[1,N]\cap B(c,r,\del)\right|\leq 2\del
\] 
where $B(c,r,\del)$ is the set of all natural numbers $n$ so that $|m-cn-r|<\del$ for some integer $m$, and $|\Gam|$ stands for the cardinality of a finite set $\Gam$. Set 
\[
C_\del=Q_k'(\gam_k(\del))=\min_{y\in[\del,1]} Q_k'(\gam_k(y))>0
\]
and
\[
D=|q_0(r)-P_0|=|d|.
\]
Note that there exists a constant $B_\del>0$ so that for any sufficiently large $N$ and 
$n\in[\del N,N]$ we have 
\[
|x_N(n/N)|\leq B_\del N^{-1}
\]
where $X_N(n/N)$ was defined in Section \ref{Sec hom decomp}. 
Relying now on Lemma \ref{Lemma}, we obtain that for any $n\in[\del N,N]\cap B(c,r,\del)$ and  $m\in\bbN$ we have
\[
|q_N(m)-p_N(n)|
\geq \frac12 C_\del N^{k-1}(\del-B_\del N^{-1})-D.
\]
Therefore, for any sufficiently large $N$ we have
\[
|\{n\in[1,N]:\,d_{N}(n)\geq A_{\del} N^{k-1}\}|\geq (1-3\del)N
\]
where $A_\del$ is a constant which depends only on $\del$, and thus the difference of $q$ and $p$ explode.
Next, suppose that $c=u/v$ is rational and that $r$ is irrational. Then for any $n$ and $m$ we have
\[
|m-cn-r|=|v|^{-1}\cdot|rv-(mv-un)|\geq|v|^{-1}\inf_{l\in\bbN}|rv-l|:=\del_0>0,
\] 
which as in the previous case is enough in order to derive that the differences of $q$ and $p$ explode. Note that when $c$ and $r$ are not both rational then the polynomials $q$ and $p$ are not $\bbQ$-equivalent.

Next, suppose that $c=u/v$ and $t=w/t$ are rational and that $C_u$ does not equal identically $0$ for some $2\leq u\leq k-1$. Then $q$ and $p$ are not $\bbQ$-equivalent. Since $c$ and $r$ are rational,  either $m=cn+r$ or 
$
|m-cn-r|=|tv|^{-1}|mtv-utn-vw|\geq |tv|^{-1}:=\del_1>0.
$  
In the case when $|m-cn-r|\geq\del_1$ and $n\in[N\del, N]$, as in the first part of this proof, we have 
\[
|q_N(m)-p_N(n)|\geq \frac12 C_\del N^{k-1}(\del_1-B_\del N^{-1})-D
\]
Assume now that $m=cn+r$ (i.e. $m=N\gam_k(n/N)+R_k(n/N)$) and $n\in[N\del, N]$. Let $s+1<k$ and $\Gam_{N,\del}$ be as in Lemma \ref{lemma}. Then 
for any $n\in\Gam_{N,\del}$ we have
\[
|q_N(m)-p_N(n)|\geq \frac12 B_\del N^{k-1-s}-D.
\]
for some constant $A_\del>0$, which completes the proof that the differences of $p$ and $q$ explode in the case considered above.

Finally, suppose that $c=u/v$ and $t=w/t$ are rational and that $C_u\equiv 0$ for any $u$ (i.e. that $p$ and $q$ and $\bbQ$-equivalent).
Then $H_{N,y}(0)=\sum_{u=2}^{k-1}N^{-(u-1)}C_u(y)=0$ for any $y$.
As in the previous cases covered in the this proof, when $m\not=cn+r$ then $|m-cn-r|\geq \del_1>0$ and so, by Lemma \ref{Lemma}, if
$Q_k'(\gam_k(cn/N))>2A_1N^{-1}$ we have
\[
|q_N(m)-p_N(n))|\geq \frac12 Q_k'(cn/N) N^{k-1}(\del_1-BN^{-1})-D
\]
for some constant $B>0$. 
Here $D=|q_0(r)-P_0|=|d|$ and  we also used (\ref{A function}). Note that $Q_k'(\gam_k(y))\geq C_\del>0$ when 
$y=n/N\in[\del,1]$. 
On the other hand, if $m=cn+r$ then 
$m=N\gam_k(y)+R_k(y)$ and so
\begin{eqnarray*}
q_N(m)-p_N(n)=d+Q_k'(\gam_k(y))N^{k-1}H_{N,y}(m-N\gam_k(y)-R_k(y))\\=
q_0(r)-P_0+Q_k'(\gam_k(y))N^{k-1}H_N(0)=q_0(r)-P_0.
\end{eqnarray*}
\end{proof}

\subsection{Application II: ``fractionaly related" polynomials}\label{frac sec}
In this section we will give several classes of examples for polynomials with exploding differences so that $\gam_k(y)$ is not a linear function of $y$. We will rely on the following

\begin{lemma}\label{Number lemma}
Let $a$ and $b$ be positive  coprime integers so that $a>b$ and  $\al_b, \alpha_{b+1},...,\al_{b}$ be integers so that $|\al_b|=1$. For each fixed $N$, consider the equation
\begin{equation}\label{SpecEq}
m^a=\sum_{j=b}^a\al_j n^j N^{a-j}
\end{equation}
over the positive integers. Then for any  $\del>0$ there exists a constant $N_\del$ and a set
$\Gam_{N,\del}\subset[1,N]$ whose cardinality does not exceed $\del N$ so that  for any $N>N_\del$ and $n\in[1,N]\setminus\Gam_{N,\del}$ 
there is no natural number $m$ such that (\ref{SpecEq}) holds true.
\end{lemma}

\begin{proof}
Fix a natural number $N$.
For any $v$ that divides $N$ set $N_v=N/v$ and
\[
U_{v,N}=\{1\leq n\leq N:\,\gcd(n,N)=v\}=\{1\leq vn_v\leq N:\,\,\gcd(n_v,N_v)=1\}
\]
and note that the cardinality $|U_{v,N}|$ of $U_{v,N}$ is just $\varphi(N_v)$, where $\varphi$ is the Euler totient function.
Therefore
\begin{equation}\label{Composition}
\sum_{v|N}|U_{v,N}|=\sum_{v|N}\varphi(N_v)=N.
\end{equation}
Denote by $U_{v,N}'$ the set of all members of $U_{v,N}$ for which there exists $m$ satisfying (\ref{SpecEq}).
In order to find sets $\Gam_{N,\del}$ with the properties described in the statement of the lemma, it is enough to show that  
\begin{equation}\label{Goal}
\lim_{N\to\infty}\frac1{N}\sum_{v|N}|U_{v,N}'|=0
\end{equation}
where $|U_{v,N}'|$ is the cardinality of $U_{v,N}'$,
since the expression inside the limit is just $N^{-1}$ times the number of $n$'s in $[1,N]$ for which (\ref{SpecEq}) holds with some $m$.

Let $n\in U_{v,N}'$ for some $v$ and. Then (\ref{SpecEq}) holds true for some $m$, and this $m$ must divide $v$. Therefore  (\ref{SpecEq}) also holds true with $n_v=n/v$ and $m_v=m/v$ and $N_v$ in place of $n,m$ and $N$, respectively.
Let $p$ be a prime number that divides $n_v$ and let $e$ be the largest power of $p$ so that $p^e$ divides $n_v$. Then $eb$ is the largest power of $p$ that divides $n_v^b$. Write 
$eb=ka+w$ for some $k$ and $0\leq w<a$. Then $p^{ka+w}$ divides $m_v^a$ and so, if $w\not=0$, then $p^{k+1}$ divides $m_v$ which implies that $p^{ak+a}$ divides $n_v^b$ since 
\[
\sum_{j=b}^a\al_j n_v^jN_v^{a-j}=n_v^{b}\big(\al_bN_v^{a-b}+n_v\sum_{j=b+1}^an_v^{j-1}N_v^{a-j}\big)
\]
and the second factor of the right hand side is not divisible by $p$ (since $\gcd(n_v,N_v)=1$ and $|\alpha_b|=1$).
 This is clearly a contradiction since $ak+a>eb$. Therefore, $w=0$ and so $eb=ka$ for some $k$. But $a$ and $b$ are coprime, and hence $k$ must have the form $k=k'b$ for some integer $k'$. Therefore we can write $e=k'a$, namely $n$ must have the form $n=vz^a$ for some integer $z$. Next, recall the following inequality (see Theorem 15 in \cite{Ross}), 
\begin{equation}\label{R(M)}
\varphi(M)\geq R(M):=\frac{M}{e^\gam \ln\ln M+3\big(\ln\ln M\big)^{-1}},\,\, M>2
\end{equation}
where $\gam$ is Euler's constant. The equality $n_v=z^a$ clearly implies that $\gcd(z,N_v)=1$ and therefore $vz$ is a member of $U_{v,N}$.  The  map $n\to z$ from $U_{v,N}'$ to $U_{v,N}$, where $n=vz^a$ is one to one, and its image is contained in the interval $[1,N_v^{\frac1a}]$ (since $z\leq N_v^{\frac1a}$), and hence $|U_{v,N}'|\leq N_v^{\frac1a}$. When $N_v>2$ then, with $c=3\big(\ln\ln 3 \big)^{-1}$, applying (\ref{R(M)}) we derive that
\[
|U_{v,N}'|\leq N_v^{\frac1a}=R(N_v)\cdot (e^{\gam}\ln\ln N_v+c) N_v^{-(1-\frac1a)}
\leq |U_{v,N}|(e^{\gam}\ln\ln (N_v)+c) N_v^{-(1-\frac1a)}
\]
which together with (\ref{Composition}) yields (\ref{Goal}).
\end{proof}

The following result follows now from Lemma \ref{Lemma} and \ref{Number lemma}.
\begin{corollary}\label{fractional corollary}
Suppose that $R_k\equiv0$ and that $\gam_k$ has the form 
\[
\gam_{k}(y)=\Big(\sum_{j=b}^a\al_j y^j\Big)^{\frac1a}
\]
where $a,b$ and $\al_j$ satisfy the conditions of Lemma \ref{Number lemma}. 
Let $s\leq k-1$ be so that 
$C_u\equiv0$ for any $2\leq u\leq s$. Suppose, in addition, that $a<s$. 
Then the polynomials $q$ and $p$ have exploding differences. 
\end{corollary}

\begin{proof}
Observe that $m=N\gam_k(n/N)$ if and only if the equation (\ref{SpecEq}) holds true. Therefore, by Lemma \ref{Number lemma}, in order to complete the proof of Corollary \ref{fractional corollary}, it is sufficient to show that for any $n$ and $m$ so that $m\not=N\gam_N(n/N)$ we have
\begin{equation}\label{Bound 1}
|q_N(m)-p_N(n)|\geq Q_k'(c\gam_k(y))N^{k-1-a}
\end{equation}
where $y=n/N$ and $c$ is some constant. 
Since $x_N(y)\leq BN^{-s}$ when $\gam_k(y)>2A_1N^{-1}$, where $B>0$ is some constant, then by Lemma \ref{Lemma} in order to show that (\ref{Bound 1}) holds true,
it is enough to show that for any $n$ and $m$ so that $m\not=N\gam_k(n/N)$ we have
\[
|m-N\gam_k(y)|\geq CN^{-(a-1)}.
\]
In order to prove the latter inequality, we define the polynomial $W(x)$ by 
\[
W(x)=x^{a}-\sum_{j=b}^a\al_jn^j N^{b-j}.
\]
Then $W(N\gam_k(y))=0$ and
\[
|W(m)|=|m^a-\sum_{j=b}^a\al_jn^j N^{b-j}|\geq1
\]
since we have assumed that $m\not=N\gam_k(y)$.
Therefore, by the mean value theorem, for some $\xi$ between $m$ and $N\gam_k(y)$,
\begin{eqnarray*}
|m-N\gam_k(y)|=a^{-1}\xi^{-(a-1)}|W(m)-W(z)|=a^{-1}\xi^{-(a-1)}|W(m)|\\\geq 
a^{-1}\xi^{-(a-1)}\geq a^{-1}N^{-(a-1)}
\end{eqnarray*}
and the proof of the corollary is complete.
\end{proof}

The following result also follows
\begin{corollary}\label{N cor}
Suppose that the conditions of Proposition \ref{Prop1} hold and that 
$
P(n,N)=\sum_{j=b}^a\al_j n^jN^{a-j}
$
and $Q(m,N)=m^{a}$ for some positive coprime integers $a>b$ and integers $\al_b,...,\al_a$ so that $|\al_b|=1$ (see Corollary \ref{N cor}). Then the polynomials $q$ and $p$ have exploding differences.
\end{corollary}

\begin{remark}
Suppose that $R_k\equiv0$ and let $s$ be as in Corollary \ref{fractional corollary}.
Assume also that $\gam_k(y)$ has the form $\gam_k(y)=K^{-1}E(y)$ for some polynomials $K$ and $H$ with integer coefficients whose degrees do not exceed $d$, for some $d<s$. Consider the polynomial $W(x)=K(x/N)-E(n/N)$. Then $W(N\gam_k(n/N))=0$ and therefore, applying the mean value theorem yields that
\[
|m-NK^{-1}E(n/N)|=\big(W'(\xi)\big)^{-1}|W(m)|
\]
for some $\xi$ between $m$ and $N\gam_k(y)$ (which is of order $N$ when $n,m\in[\del N,N]$). Notice that either $W(m)=0$ or 
\[
|W(m)|\geq CN^{-d}
\]
for some constant $C>0$ (when $1\leq m\leq N$). We conclude that for any $\del>0$ there exists a constant $C_\del>0$ so that for any $n,m\in[N\del,N]$ we either have $m=NK^{-1}E(n/N)$ or
\[
|m-NK^{-1}E(n/N)|\geq C_\del N^{-d-1}.
\]
The difficulty in using the above estimates in order to determine whether the polynomials $q$ and $p$ have exploding differences arises here in determining for which $n$'s there exists a solution $m$ to the equation $m=NK^{-1}E(n/N)$. When $\gam_k$ satisfies the conditions of Corollary \ref{fractional corollary} then we used Lemma \ref{Number lemma}, but for general polynomials  $K$ and $H$ it does not seem possible to show that the equation $m=NK^{-1}E(n/N)$ does not have a solution for any $n\in[\del N,N]\setminus\Gam_{N,\del}$ for sets $\Gam_{N,\del}$ whose cardinality does not exceed $\del N$.
%
\end{remark}



\section{Expectation estimates}
In this section we describe two results from \cite{book} which are key ingredients in the
 the proofs of Theorems \ref{LLN}, \ref{VarThm} and \ref{CLTthm}. 

Let $U_i,\,i=1,2,...,k$ be  $d_i$-dimensional random vectors defined on the
probability space $(\Om,\cF,P)$ from Section \ref{Sec1}, and
$\{\cC_j: 1\leq j\leq s\}$ be a partition of $\{1,2,...,k\}$. 
Consider the random vectors $U(\cC_j)=\{U_i: i\in\cC_j\}$, $j=1,...,s$, 
and let 
\[
U^{(j)}(\cC_i)=\{U_i^{(j)}: i\in\cC_j\},\,\, j=1,...,s
\]
be independent copies of the $U(\cC_j)$'s.
For each  $1\leq i\leq k$ let $a_i\in\{1,...,s\}$ be the unique index
such that $i\in\cC_{a_i}$, and 
for any   Borel function $H:\bbR^{d_1+d_2+...+d_k}\to\bbR$ set
\begin{equation}\label{ExpDiffDef}
\cD(H)=\big|\bbE H(U_1,U_2,...,U_k)-\bbE H(U_1^{(a_1)},U_2^{(a_2)},...,U_k^{(a_k)})\big|.
\end{equation}
Of course, the above quantity is defined only when the expectations inside the absolute value exist.

 The first result appears in \cite{book} as Corollary 1.3.14, and it formulation is as follows:

\begin{proposition}\label{1.3.14}
Let
$H:\bbR^{d_1+d_2+...+d_k}\to\bbR$ be a function satisfying (\ref{F Bound}) and 
(\ref{F Hold})  with $H$ in place of $F$ and with 
$u_i$'s in place of $x_i$'s.  Let $q,w>1$ and $v>0$ be such that 
\[
\frac 1w>\frac \iota v+\frac\ka q
\]
and set $r_i=[\frac13(m_i-n_{i-1})], i=2,...,k$, $r_1=r_2$ and $r_{k+1}=r_k$. Suppose that 
$\|U_i\|_{L^w}<\infty$ for any $1\leq i\leq k$. 
Then $\cD(H)$is well defined and 
\begin{equation}\label{l.h.s. CoeExpect3}
\cD(H)\leq 6R_0\Lambda_{w,\ka}
\end{equation}
where 
\[
\Lambda_{w,\ka}=\big(\sum_{i=2}^k\phi(r_i)\big)^{1-\frac1b}+
\sum_{i=1}^k\big(\big\|U_i-\bbE[U_i|\cF_{m_i-r_i,n_i+r_{i+1}}]\big\|_{L^q}\big)^\ka
\] 
$R_0=K(1+k g_m^\iota)$ and
$
\,g_m=\max\{\|U_i\|_{L^m}: 1\leq i\leq k\}
$. Here the $\sigma$-algebras $\cF_{k,l}$ are the ones specified in Section \ref{Sec1}.
\end{proposition}

Next, recall that the $\alpha$-dependence coefficients of any two sub-$\sig$-algebras $\cG$ and $\cH$ of $\cF$
is given by
\begin{equation}\label{general alpha mixing}
\al(\cG,\cH)=\sup\big\{|P(\Gam\cap\Del)-P(\Gam)P(\Del)|:
 \Gam\in\cG,\,\Del\in\cH\big\}.
\end{equation}
Another result we will need is the following proposition, which is stated in \cite{book} as Corollary 1.3.11:

\begin{proposition}\label{1.3.11}
Suppose that each $U_i$ is $\cF_{m_i,n_i}$-measurable, where $n_{i-1}<m_i\leq n_i<m_{i+1}$, 
$i=1,...,k$, $n_0=-\infty$ and $m_{k+1}=\infty$.  
Then, for any bounded
Borel function $H:\bbR^{d_1+d_2+...+d_k}\to\bbR$,
\begin{equation*}
\cD(H)\leq 4\sup|H|\sum_{i=2}^k\phi(m_i-n_{i-1})
\end{equation*}
where $\sup|H|$ is the supremum of $|H|$. In particular, when $s=2$ then
\[
\al\big(\sig\{U(\cC_1)\},\sig\{U(\cC_2)\}\big)\leq4\sum_{i=2}^k\phi(m_i-n_{i-1})
\]
where $\sig\{X\}$ stands for the $\sig$-algebra generated by a random variable 
$X$ and $\al(\cdot,\cdot)$ is given by (\ref{general alpha mixing}).
\end{proposition}

In fact, Proposition \ref{1.3.14} follows from Proposition \ref{1.3.11} using standard approximations which involve 
 the H\"older and the Markov inequalities (note the function $H$ in Proposition \ref{1.3.14} is locally H\"older continuous).
The proof of Proposition \ref{1.3.14} is based on the following lemma (see Lemma 1.3.10 in \cite{book} or Lemma 3.11 in \cite{Hafouta ECP}).

\begin{lemma}\label{thm3.11-StPaper}
Let $\cG_1,\cG_2\subset\cF$ be two sub-$\sigma$-algebras of $\cF$ and for $i=1,2$ let
$V_i$ be a $\bbR^{d_i}$-valued random $\cG_i$-measurable vector with distribution $\mu_i$.  Set
$d=d_1+d_2$, $\mu=\mu_1\times\mu_2$, denote by $\ka$
the distribution of the random vector $(V_1,V_2)$
and consider the measure $\nu=\frac12(\ka+\mu)$.  Let $\cB$  be the Borel $\sigma$-algebra on $\bbR^d$ and 
$H\in L^\infty(\bbR^d,\cB,\nu)$. Then $\bbE[H(V_1,V_2)|\cG_1]$ and $\bbE H(v,V_2)$ exist for $\mu_1$-almost
any $v\in\bbR^{d_1}$ and
\begin{equation}\label{Meas-StPaper}
|\bbE[H(V_1,V_2)|\cG_1]-h(V_1)|
\leq2\|H\|_{L^\infty(\bbR^d,\cB,\nu)}\phi(\cG_1,\cG_2),\,\,P-a.s.
\end{equation}
 where $h(v)=\bbE H(v,V_2)$ and a.s. stands for almost surely.
\end{lemma}
For the sake of completeness, let us explain the idea behind the proof of Proposition \ref{1.3.11}.
In the circumstances of Lemma \ref{thm3.11-StPaper}, it follows by induction on $k$ that 
\begin{eqnarray}
|\bbE H(U_1,U_2,...,U_k)-\int H(u_1,u_2,...,u_k)d\nu_1(u_1)d\nu_2(u_2)...d\nu_k(u_k)|
\nonumber\\
\leq 2\sup|H|\sum_{i=2}^v\phi(m_i-n_{i-1}).\hskip2cm\label{Expec2-StPaper}
\end{eqnarray}
where $\nu_i$ is the distribution of $U_i$.
The proof of Proposition \ref{1.3.11} is carried out by induction on the number of blocks $s$, where the main ingredient in the induction step is (\ref{Expec2-StPaper}) (applied with various functions appearing in the proof).

\section{First four moments growth and the SLLN}\label{Sec3}
\subsection{Linear growth of the variance}
For each $n$ and $N$ set 
\[
F_{n,N}=F(\xi_{q_{1}(n,N)},...,\xi_{q_\ell(n,N)}).
\]
We begin with noting that
by Assumption \ref{Moment Ass} and (\ref{F Bound}) there exists a constant $B>0$ so that 
\begin{equation}\label{B bound}
\sup_{n,N}\|F_{n,N}\|_{L^w}\leq B
\end{equation}
where $w$ comes from Assumption \ref{Moment Ass}.
Next, for each $l$ set 
\[
a(l)=\big((\phi(l/3))^{1-\frac 2w}+(\beta_q(l/3))^\ka\big)
\]
here $w$ and $\te$ come from Assumptions \ref{Moment Ass}  and  \ref{Mix2}.
In the circumstances of Theorems \ref{VarThm} and \ref{CLTthm}, we have $\te>\frac{4w}{w-2}$ which implies that
$\sum_{l=0}^\infty (l+1)a(l)<\infty$.

Our first result shows, in particular, that the variance of $S_N=\cS_N(1)$ behaves as the expectation of $S_N^2$:
\begin{lemma}\label{expectation lemma}
Suppose that Assumption \ref{Moment Ass} holds true and that
$\sum_{u=0}^\infty a(u)<\infty$.
Then there exists a constant $C>0$ so that for any $N\geq1$,
\[
\sum_{n=1}^N\big|\bbE F_{n,N}\big|\leq C.
\]
\end{lemma}

\begin{proof}
Fix $N\geq1$ and $u\geq0$ and set $k^*=\max\{\deg q_i:\,1\leq i\leq\ell\}$. Then there exists at most $2\ell^2 k^*$ number of natural $n$'s so that 
\[
\del_N(n,n):=\min_{1\leq i\not=j\leq \ell}|q_i(n,N)-q_j(n,N)|=u.
\]
Let $1\leq n\leq N$ be so that $\del_N(n,n)=u$.
Then there exists  a permutation $\sig$ of $\{1,...,\ell\}$ so that 
\[
q_{\sig(i)}(n,N)+u\leq q_{\sig(i+1)}(n,N)\,\,\text{ for any }\,\,i=1,2,...,\ell-1.
\] 
Set $U_i=\xi_{\sig(i)}$ and $H(x_1,...x_\ell)=F(x_{\sig^{-1}(1)},...,x_{\sig^{-1}(\ell)})$. 
Applying Proposition \ref{1.3.14}  with the function $H$ and the random vectors $U_i$, when $\del_N(n,n)=u$ then
\[
|\bbE F_{n,N}|\leq R_0\big((\phi(u/3))^{1-\frac 1b}+(\beta_q(u/3))^\ka\big)\leq R_1a(u)
\]
where $R_0$ and $R_1$ are some constants and we also used (\ref{F bar}). Note that, in the terminology of  Proposition \ref{1.3.14}, we used the partition of the index set $\{1,...,\ell\}$ into points: $\cC_i=\{i\}$.
We conclude that
\[
\sum_{n=1}^N\big|\bbE F_{n,N}\big|\leq 2\ell^2 k^*\sum_{u=0}^\infty a(u):=C<\infty
\]
and the proof of the lemma is complete.
\end{proof}

Now we will show that the variance of $S_N$ grows at most linearly fast in $N$. In fact, we will prove the following
\begin{lemma}\label{VarLemma}
Suppose that Assumption \ref{Moment Ass} holds true and that
$\sum_{l=0}^\infty a(l)<\infty$. Then there exists a constant $C>0$ so that for any positive integers $n_1<n_2<...<n_M$ and $N\in\bbN$ we have
\[
\bbE(\sum_{k=1}^M F_{n_k,N})^2\leq CM
\]
where for each random variable $Z$ we set $\bbE Z^2=\bbE[Z^2]$.
\end{lemma}

\begin{proof}
Fix some $N$ and let $n_1<...<n_M$ be positive integers.
Using Lemma \ref{expectation lemma}, it is enough to show that the variance of the sum 
$\sum_{k=1}^MF_{n_k,N}$ is bounded by $CM$ for some constant $C$ which does not depend on $N$ and the choice of $n_i$.
For each $n$ and $m$ set
\begin{equation}\label{d N def}
d_N(n,m)=\min_{1\leq i,j\leq\ell}|q_i(n,N)-q_j(m,N)|.
\end{equation}
For each $l\geq0$ set 
\[
\Gam_{n,N,l}=\{m:\,\,d_N(n,m)=l\}.
\]
Then, since each $q_i(x,N)$ is a  polynomial function of the variable $x$ (for any fixed $N$), whose degree is bounded by some constant which does not depend on $N$, we have 
\[
|\Gam_{n,N,l}|\leq A
\]
for some constant $A$ which does not depend on $n,N$ and $l$. Here $|\Gam_{n,N,l}|$ denotes the cardinality of the set $\Gam_{n,N,l}$. Now we can write
\begin{equation}\label{Var1}
\text{Var}\big(\sum_{k=1}^MF_{n_k,N}\big)=
\sum_{k=1}^M\sum_{l=0}^\infty\sum_{n_s\in\Gam_{n_k,N,l}\cap[1,N]}\text{Cov}\big(F_{n_k,N},F_{n_s,N}\big).
\end{equation}
Let  $n,m\in\bbN$ and $l\geq0$ be so that $d_N(n,m)=l$, and consider the sets 
\[
\Gam_1=\Gam_1(n)=\{q_i(n,N):\,1\leq i\leq\ell\}\,\text{ and }\,\,\Gam_2=\Gam_1(m)=\{q_i(m,N):\,1\leq i\leq\ell\}.
\]
Then, there exist disjoint sets $Q_1,...,Q_L$, $L\leq 2\ell+1$ so that 
\begin{equation}\label{Qi's}
\Gam:=\Gam_1\cup\Gam_2=\bigcup_{i=1}^{L}Q_i,
\end{equation}
each one of the $Q_i$'s is contained in either $\Gam_1$ or $\Gam_2$ and for any $q_i\in Q_i$ and $q_{i+1}\in Q_{i+1}$, $i=1,2,..,L-1$ we have
\[
q_i+l\leq q_{i+1}.
\]
Consider now the random vectors $U_1,...,U_L$ given by 
\[
U_i=\{\xi_{j}:\,j\in Q_i\}.
\]
Consider the partition $\{\cC_1,\cC_2\}$ of the index set $\{1,...,L\}$ given by
\[
\cC_i=\bigcup_{j:Q_j\subset\Gam_i}Q_j,\,\,i=1,2.
\]
Then $F_{n,N}$ is a function of $\{U_j:\,j\in\cC_1\}$ and $F_{m,N}$ is a function of $\{U_j:\,j\in\cC_2\}$.
Applying Proposition \ref{1.3.14}  with the function $H(x,y)=F(x)F(y)$ we obtain that for any $n=n_k$, $l\geq0$ and $m=n_s\in\Gam_{n_k,N,l}$,
\[
|\text{Cov}\big(F_{n_k}, F_{n_s}\big)|\leq R_0\big((\phi(l/3))^{1-\frac 2b}+(\beta_q(l/3))^\ka\big)=R_0a(l)
\]
where $R_0>0$ is some constant. Therefore, using (\ref{Var1}), we obtain that 
\begin{equation*}
\text{Var}\Big(\sum_{k=1}^MF_{n_k,N}\Big)\leq \sum_{k=1}^M\sum_{l=0}^\infty\sum_{n_s\in\Gam_{n_k,N,l}\cap[1,N]}|\text{Cov}(F_{n_k,N},F_{n_s,N})|\leq \big(AR_0\sum_{l=0}^\infty a(l)\big)M
\end{equation*}
and the proof of the lemma is complete.
\end{proof}

\subsection{Fourth moments and the strong law of large numbers}
We will prove here the following 
\begin{lemma}\label{4 lemma}
Set $b(l)=(\phi(l))^{1-\frac 4w}+(\beta_q(l))^\ka$ and assume that
$\sum_{l=1}^\infty lb(l)<\infty$. Then there exists a constant $C>0$ so that for any positive integers $n_1<n_2<...<n_M$ and $N\in\bbN$ we have
\[
\bbE\big(\sum_{k=1}^MF_{n_k,N}\big)^4\leq CM^2.
\]
\end{lemma}
Relying on Lemma \ref{4 lemma} and  using the Markov inequality, we obtain that 
\[
P\big\{N^{-1}|S_N|\geq N^{-1/8}\big\}\leq\frac{\|S_N\|^4_{L^4}}{N^{7/2}}\leq C N^{-3/2}
\]
which together with the Borel-Cantelli lemma yields that with probability one we have
\[
\lim_{N\to\infty}\frac1{N}S_N=0=\bar F.
\]

\begin{proof}[Proof of Lemma \ref{4 lemma}]
Relying on Lemmas \ref{expectation lemma} and \ref{VarLemma}, it is enough to prove that 
\[
\text{Var}\big((\sum_{k=1}^M \overline{F_{n_k,N}})^2\big)\leq CM
\]
for some $C$ which do not depend on $N$ and the choice of $n_1,...,n_M$, where $\bar X=X-\bbE X$ for any random variable $X$. For any $u_1,u_2,v_1,v_2$ set 
\[
b_N(u_1,u_2,v_1,v_2)=\text{Cov}(\overline{F_{v_1,N}}\cdot\overline{F_{v_2,N}},\overline{F_{u_1,N}}\cdot\overline{F_{u_2,N}}).
\]
For each $u$ and $N$ set 
\[
\Gam_{u,N}=\{q_i(u,N):\,1\leq i\leq\ell\}.
\]
For any two sets $A_1,A_2\subset\bbR$ set $\text{dist}(A_1,A_2)=\inf\{|a_1-a_2|:\,a_i\in A_i\}$.
Since the $q_i$'s are polynomials, there exists a constant $A>0$, which does not depend on $N$, so that 
for each nonnegative integers $k_1$ and $k_2$ and positive integers $u_1$ and  $u_2$, there exist at most $A$ pairs $(v_1,v_2)$ of positive integers such that 
\[
k_1=\text{dist}(\Gam_{u_1,N}\cup\Gam_{u_2,N},\Gam_{v_1,N}\cup\Gam_{v_2,N})
\]
and, if $k_1=\text{dist}(\Gam_{u_1,N}\cup\Gam_{u_2,N},\Gam_{v_1,N})$ then 
\[
k_2=\text{dist}(\Gam_{u_1,N}\cup\Gam_{u_2,N}\cup\Gam_{v_1,N},\Gam_{v_2,N})
\]
while when $k_1>\text{dist}(\Gam_{u_1,N}\cup\Gam_{u_2,N},\Gam_{v_1,N})$ (and so $k_1=\text{dist}(\Gam_{u_1,N}\cup\Gam_{u_2,N},\Gam_{v_2,N})$) then 
\[
k_2=\text{dist}(\Gam_{u_1,N}\cup\Gam_{u_2,N}\cup\Gam_{v_2,N},\Gam_{v_1,N}).
\]
Let us denote the set of all these indexes $(v_1,v_2)$ by $\Del(u_1,u_2,k_1,k_2,N)$. Note that given $N,u_1,u_2,v_1,v_2$ there exist  $k_1$ and $k_2$ so that $(v_1,v_2)\in\Del(u_1,u_2,k_1,k_2,N)$. 
Next, we claim that for any $u_1,u_2$ and $(v_1,v_2)\in\Del(u_1,u_2,k_1,k_2,N)$ we have
\begin{equation}\label{b est}
|b_N(u_1,u_2,v_1,v_2)|\leq R_0\tau(\max(k_1,k_2))
\end{equation}
where $\tau(l)=b(l/3)$, $R_0$ is some constant and $b(\cdot)$ was defined in the statement of the lemma. 
Indeed, first consider the case when $k_1\geq k_2$, and set $\Del_1=\Gam_{u_1,N}\cup\Gam_{u_2,N}$ and $\Del_2=\Gam_{v_1,N}\cup\Gam_{v_2,N}$.
Then we can write
\[
\Del_1\cup \Del_2=\bigcup_{i=1}^LQ_i,\,L\leq 4\ell+4
\]
where each $Q_i$ is subsets of either $\Del_1$ or $\Del_2$, and the $Q_i$'s satisfy all the conditions appearing right after (\ref{Qi's}) with $k_1$ in place of $l$.  Consider the random vectors $U_1,...,U_L$ given by $U_i=\{\xi_{j}:\,j\in Q_i\}$
and  the partition $\{\cC_1,\cC_2\}$ of the index set $\{1,...,L\}$ given by
\[
\cC_i=\bigcup_{j:Q_j\subset\Gam_i}Q_j,\,\,i=1,2.
\]
Then $\overline{F_{u_1,N}}\cdot\overline{F_{u_2,N}}$ is a function of $\{U_j:\, j\in\cC_1\}$ and 
 $\overline{F_{v_1,N}}\cdot\overline{F_{v_2,N}}$ is a function of $\{U_j:\, j\in\cC_2\}$. 
Applying  Proposition \ref{1.3.14}  with the function $H(x,y,z,w)=F(x)F(y)F(z)F(w)$ we obtain
that 
\begin{equation*}
|b_N(u_1,u_2,v_1,v_2)|\leq R_0\tau(k_1)
\end{equation*}
for some $R_0$, which completes the proof of (\ref{b est}) in the above case.

Next, consider the case when $k_1<k_2$ and  $k_2=\text{dist}(\Gam_{u_1,N}\cup\Gam_{u_2,N}\cup\Gam_{v_1,N},\Gam_{v_2,N})$. Set  
$\Del_1=\Gam_{u_1,N}\cup\Gam_{u_2,N}\cup\Gam_{v_1,N}$ and $\Del_2=\Gam_{v_2,N}$. Then $\text{dist}(\Del_1,\Del_2)=k_2$.
As in the previous case, applying  Proposition \ref{1.3.14}  yields that 
\[
\left|\bbE[\overline{F_{u_1,N}}\cdot\overline{F_{u_2,N}}\cdot\overline{F_{v_1,N}}\cdot\overline{F_{v_2,N}}]
-\bbE[\overline{F_{u_1,N}}\cdot\overline{F_{u_2,N}}\cdot\overline{F_{v_1,N}}]\cdot\bbE[\overline{F_{v_2,N}}]\right|\leq R_1\tau(k_2)
\]
for some constant $R_1$. An additional application of this corollary yields
\[
\left|\bbE[\overline{F_{v_1,N}}\cdot\overline{F_{v_2,N}}]
-\bbE[\overline{F_{v_1,N}}]\cdot\bbE[\overline{F_{v_2,N}}]\right|\leq R_2a(k_2)\leq R_2\tau(k_2)
\]
where $a(\cdot)$ was defined at the beginning of this section and $R_2$ is some constant. Taking into account that $\bbE[\bar X]=0$ for any random variable $X$, combining the above estimates with (\ref{B bound}) we obtain that 
\begin{equation*}
|b_N(u_1,u_2,v_1,v_2)|\leq R_0\tau(k_2)
\end{equation*}
for some constant $R_0>0$.
The proof of (\ref{b est}) in the case when  $k_1<k_2$ and  $k_2=\text{dist}(\Gam_{u_1,N}\cup\Gam_{u_2,N},\Gam_{v_1,N})$ proceed exactly in the same way by changing the roles for $v_1$ and $v_2$.

Finally, applying (\ref{b est}) we obtain that 
\begin{eqnarray*}
\text{Var}\big((\sum_{k=1}^M \overline{F_{n_k,N}})^2\big)=\sum_{1\leq i,j\leq M}\sum_{k_1,k_2=0}^{\infty}\,\,
\sum_{(n_{i_1},n_{j_1})\in\Del(n_i,n_j,k_1,k_2,N)}b_N(n_i,n_j,n_{i_1},n_{i_2})\\\leq 
A\sum_{1\leq i,j\leq M}\sum_{k_2=0}^\infty\sum_{k_1=0}^{k_2}\tau(k_2)\\+
A\sum_{1\leq i,j\leq M}\sum_{k_1=0}^\infty\sum_{k_2=0}^{k_1}\tau(k_1)=
2R_0AM^2\sum_{k}(k+1)b(k)\leq CM^2
\end{eqnarray*}
where $C$ is some constant, and the proof of the lemma is complete.
\end{proof}

\begin{remark}
Relying on conditional expectation type estimates, it is possible to prove Lemma \ref{4 lemma} similarly to Chapter 3 in \cite{book}, using the functions $F_{\ve,i}$ defined in Section \ref{Var Sec}. Moreover, it is also possible to derive this lemma using the method of cumulants, similarly to Section 6 in \cite{HafMD}. In fact, the arguments leading to the comulants estimates obtained in \cite{HafMD} can be modified to the setup of this paper, which means that we can also obtain moderate deviations theorems and some concentration inequalities for the sums $S_N$.
\end{remark}

\section{Limiting covariances}\label{Var Sec}


\subsection{Ordering and decomposition}


Consider the homogeneous decomposition
\begin{equation}\label{Hom decomp}
q_i(n,N)=\sum_{j=0}^kN^jQ_{i,j}(n/N)
\end{equation}
where for each $i$ and $j$ the polynomial $Q_{i,j}$ is of degree not exceeding $j$. 
Let $r_0$ be so for any $1\leq i,j\leq\ell$ with $\deg q_i>\deg q_j$, for any sufficiently large $N$ we have
\begin{equation}\label{Different degree ordering}
\min_{N^{r_0}\leq n,m\leq N}\big(q_i(n,N)-q_j(m,N)\big)\geq N .
\end{equation}

We will first prove the following
\begin{proposition}\label{Prop Ord 2}
There exist constants $r_1\in(0,1)$, $c>0$ and $A,B>0$, sets $B_N,\,N\geq1$ containing $[1,N^{r_0}]$ so 
 that $|B_N|\leq AN^{r_1}$ and disjoint sets $I_\ve(N)$ of the form
\begin{equation}\label{I form}
I_\ve(N)=\bigcup_{j=1}^{l_\ve}\big(a_{j,\ve}N+BN^{r_1},b_{j,\ve}N-N^{r_1}\big) \subset[1,N]
\end{equation}
 whose union cover $[1,N]\setminus B_N$, where $\ve$ ranges over all the permutations of $\{1,...,\ell\}$ and the sets  $(a_{j,\ve}N,b_{j,\ve}N)$'s are disjoint, so that for any sufficiently large $N$ and $n\in I_\ve(N)$ with $n\geq N^r$ we have 
\begin{equation}\label{Ord}
q_{\ve(1)}(n,N)<q_{\ve(2)}(n,N)<...<q_{\ve(\ell)}(n,N)
\end{equation}
and when $q_{\ve(i+1)}$ and $q_{\ve(i)}$ have the same non-linear degree then  
\begin{equation}\label{Same degree ordering}
q_{\ve(i+1)}(n,N)\geq q_{\ve(i)}(n,N)+cN^{\frac12}.
\end{equation}
\end{proposition}

\begin{proof}
We will prove the proposition by induction on the maximal degree of the polynomials $q_1,...,q_\ell$. 
When the maximal degree is $1$, i.e. when all the polynomials are linear, then exactly as in
Chapter 3 of \cite{book} there exist a finite union of intervals of the form $(a_\ve N,b_\ve N),\,\ve\in\cE_\ell$ whose union cover $[1,N]\setminus B_N$, for some set $B_N$ whose cardinality does not exceed $2\ell^2$,
 so that (\ref{Ord}) holds true for any $n\in (a_\ve N,b_\ve N)$, and 
\[
q_{\ve(i+1)}(n,N)-q_{\ve(i)}(n,N)\geq \min(n-a_\ve N-1,B_\ve N-n-1)
\]
for any $n\in(a_\ve,b_\ve)$ and $1\leq i<\ell$. Now we can just take $r_1=\frac12$ (when all of the polynomials are linear, (\ref{Different degree ordering}) is meaningless). For readers' convenience, let us repeat the details from \cite{book}. 
Write  $q_i(n,N)=p_in+q_iN,\,
i=1,...,\ell$, and assume without loss of generality that $p_1<p_2<...<p_\ell$. Since
$p_in+q_iN\geq 0$ always then $q_i\geq 0$ for $i=1,...,\ell$. Let $\cN_N$
be the set of $n\in\{ 1,2,...,N\}$ such that all $p_in+q_iN,\, i=1,...,\ell$
are distinct. For each $n\in\cN_N$ we define distinct integers $\ve_i(n,N),
\, i=1,2,...,\ell$ such that
\begin{equation}\label{3.2.17}
p_{\ve_i(n,N)}n+q_{\ve_i(n,N)}N<p_{\ve_{i+1}(n,N)}n+q_{\ve_{i+1}(n,N)}N\,\,\,
\mbox{for all}\,\,i=1,2,...,\ell-1.
\end{equation}
For each
$\ve=(\ve_1,...,\ve_\ell)\in\cE_\ell$ set 
\[
\cN_{\ve,N}=\{ n\in\{ 1,2,...,N\}:\, \ve_j(n,N)=\ve_j\,\,\mbox{for each}\,\,
j=1,...,\ell\}.
\]
Some of the sets $\cN_{\ve,N}$ can be empty and for each $n\in\cN_{\ve,N}$,
\[
p_{\ve_i}n+q_{\ve_i}N<p_{\ve_{i+1}}n+q_{\ve_{i+1}}N,\,\,\mbox{i.e.}
\]
\[
n>(q_{\ve_i}-q_{\ve_{i+1}})(p_{\ve_{i+1}}-p_{\ve_{i}})^{-1}N\,\,\mbox{if}\,\,
\ve_{i+1}>\ve_i\,\,\mbox{and}
\]
\[
n<(q_{\ve_i}-q_{\ve_{i+1}})(p_{\ve_{i+1}}-p_{\ve_{i}})^{-1}N\,\,\mbox{if}\,\,
\ve_{i+1}<\ve_i.
\]
Hence,
\[
\cN_{\ve,N}=\{ n:\, a_\ve N<n<b_\ve N\}
\]
for some (not unique) $a_\ve\geq 0$ and $b_\ve\leq 1$, $\cN_{\ve,N}$ 
are disjoint for different $\ve\in\cE_\ell$ and, clearly,
\begin{equation}\label{3.2.18}
\cN_N=\cup_{\ve\in\cE_\ell}\cN_{\ve,N}.
\end{equation}
There is always $\ve=(\ve_1,...,\ve_\ell)$ with $a_\ve=0$ and then
$\ve_1=\min\{ i:\, q_i=\min_{1\leq j\leq\ell}q_j\}$ and $\ve_\ell=
\max\{ i:\, q_i=\max_{1\leq j\leq\ell}q_j\}$.
Set $B_N=\{ 1,2,...,N\}\setminus\cN_N$. Then
\[
B_N\subset\{ (q_{i}-q_{j})(p_{j}-p_{i})^{-1}N:\,\,\mbox{for some}
\,\, i,j=1,...,\ell,\, i\ne j\},
\]
and so the cardinality of $B_N$ does not exceed $\ell^2$ (since the polynomials are linear). In Lemma 3.2.6 in \cite{book} it was proved that if $a_\ve N<n<b_\ve N$ then
\begin{equation}\label{3.2.19}
(p_{\ve_{i+1}}-p_{\ve_i})n+(q_{\ve_{i+1}}-q_{\ve_i})N\geq\min(n-a_\ve N-1,\,
b_\ve N-n-1).
\end{equation}
The proof of (\ref{3.2.19}) is as follows. Let $m_1,m_2\geq 1$ be integers satisfying $a_\ve N<n-m_1$ and
$n+m_2<b_\ve N$. Then
\[
p_{\ve_i}(n-m_1)+q_{\ve_i}N<p_{\ve_{i+1}}(n-m_1)+q_{\ve_{i+1}}N
\]
and
\[
p_{\ve_i}(n+m_2)+q_{\ve_i}N<p_{\ve_{i+1}}(n+m_2)+q_{\ve_{i+1}}N.
\]
Hence,
\[
(p_{\ve_{i+1}}-p_{\ve_i})n+(q_{\ve_{i+1}}-q_{\ve_i})N>
m_1(p_{\ve_{i+1}}-p_{\ve_i})\geq m_1
\]
if $\ve_{i+1}>\ve_i$ and if $\ve_i>\ve_{i+1}$ then
\[
(p_{\ve_{i+1}}-p_{\ve_i})n+(q_{\ve_{i+1}}-q_{\ve_i})N>
m_2(p_{\ve_{i}}-p_{\ve_{i+1}})\geq m_2,
\]
and so the assertion follows.

Now we will make the induction step.
Suppose that the proposition holds true when the maximal degree does not exceed $d$. Let $q_1,...,q_\ell$ be polynomials so that the maximal degree equals $d+1$. Let $k>\ell$ be so  that $q_1,...,q_k$ are of degree strictly less than $d+1$, and  $\deg q_i=d+1$ for any $i>k$. By the induction hypothesis, there exist constants $r_1=r_1(H)\in(0,1)$ and $A=A_H,\,c=c_H>0$ and sets $B_N=B_{H,N},\,N\geq1$ and $I_{\ve'}(N)$ satisfying all the properties described in the statement of the proposition with the polynomials $q_1,...,q_k$, where $\ve'$ ranges over all the permutations of the set $\{1,...,k\}$.

In order to complete the induction hypothesis, it is enough to show that all the results stated in the proposition hold true for the family of polynomials $q_{k+1},...,q_{\ell}$ (because of (\ref{Different degree ordering})). 
 Indeed, assume that there exist constants $r_1=r_{1,d+1}\in(0,1)$ and $c=c_{d+1}$ and sets $B_N=B_{N,d+1}$ and $I_{\ve''}(N)$ with the properties described in the statement of Proposition \ref{Prop Ord 2} for the polynomials $q_{k+1},...,q_\ell$, where $\ve''$ ranges over all the permutations of the set $\{k+1,...,\ell\}$. Take $1>r_1>\max(r_{1}(H),r_{1,d+1},r_0)$, where $r_0$ comes from (\ref{Different degree ordering}). 
Consider all the endpoints of the intervals defining the sets $I_{\ve'}(N)$ and $I_{\ve''}(N)$ which are larger than $N^{r_1}$. For any endpoint $a$, consider the set $E_a$ of all endpoints $b$ so that $|a-b|$ is sublinear in $N$. Then we can partition the set of all endpoints to disjoint sets of the form $E_a$. By omitting all the endpoints from each partition set $E_a$ except $a$, and then considering all the intervals generated by two consecutive (remaining) endpoints we get sets $I_\ve(N)$ with the desired properties (for the polynomials $q_1,...,q_\ell$).  Note that
 for any permutation of $\{1,...,\ell\}$ which does not have the form $\ve'\otimes\ve''$ we have $I_\ve(N)=\emptyset$.

Next, consider the decompositions (\ref{Hom decomp}) of the polynomials $q_{k+1},...,q_\ell$ (whose degree is $d+1$).
let $Q_{i_1,d+1},...,Q_{i_u,d+1}$ be the distinct polynomials among the polynomials $Q_{i,d+1},\,i=k+1,...,d$.
Let $0\leq y_1<y_2<...<y_s\leq 1$ be the set of all points $y$ in $[0,1]$ so that $Q_{i_j,d+1}(y)=Q_{i_j',d+1}(y)$ for some $j\not=j'$. On each interval of the form $(y_a,y_{a+1})$ we can order the polynomials $Q_{i_j,d+1}$. In fact, since the degrees of the polynomials $Q_{i_j,d+1}$ is at most $d+1$, there exists a permutation $\sig_a$ of $\{i_1,...,i_u\}$ so that for any 
\[
y\in[y_a+N^{-\frac1{2d+2}},y_{a+1}-N^{-\frac1{2d+2}}]=:J_{a,n}
\] 
and $1\leq j<u$ we have
\[
Q_{\sig_a(i_{j+1}),d}(y)>Q_{\sig_a(i_j),d}(y)+CN^{-\frac{1}2}
\]
for some constant $C>0$. Since the degree of  $q_{i},\,i>k$ is $d+1$, it follows that for any $j$ and $n$ that when $n/N\in J_{a,N}$ we have
\[
q_{\sig_a(i_{j+1})}(n,N)-q_{\sig_a(i_{j})}(n,N)\geq C_1N^{d+1-\frac12}
\]
where $C_1>0$ is some constant. In the case when all of the $Q_{i,d+1}$'s are distinct, we have completed the induction step. 
Otherwise, set $\Gam_j=\{i:Q_{i,d+1}=Q_{i_j,d+1}\}$. Then, for any $a$ and $n/N\in J_{a,N}$, $j<j'$ and $i\in \Gam_{\sig_a(j)}, i'\in\Gam_{\sig_a(j')}$ we have 
\[
q_{i'}(n,N)-q_{i}(n,N)\geq C'N^{d+1-\frac12}
\]
for some constant $C'>0$. Finally, for each $j$ and $i\in \Gam_j$ we define 
\[
\tilde q_{i}(n,N)=q_{i}(n,N)-N^{d+1}Q_{i_j,d+1}(n/N)=q_{i}(n,N)-N^{d+1}Q_{i,d+1}(n/N).
\]
Then the degrees of the polynomials $\{\tilde q_i:\,i\in\Gam_j\}$ do not exceed $d$, and so we can apply the induction hypothesis with them. By intersecting the resulting sets $I_{\ve^{(j)}}(N)$ with each interval $J_{a,N}$, where $\ve^{(j)}$ ranges over all the permutations of the set $\Gam_j$,  omitting the intervals whose length is a sublinear function of $N$
and taking a sufficiently large $r\in(0,1)$ we get disjoint sets $I_{\ve''}(N)$ which have the properties described in the statement of the proposition for the polynomials $q_{k+1},...,q_{\ell}$ (the ones with degree $d+1$), where $\ve''$ now ranges over the permutations of the set $\{k+1,...,\ell\}$.
As we have explained before, it is enough in order to complete the induction step. The proof of the proposition is complete.
 \end{proof}
 
 Let $\cE_\ell$ be the set of all permutations of the set $\{1,2,...,\ell\}$.
Let $0=s_0<s_1<s_2,...<s_{z-1}<s_{z}=\ell$ and $d_1<...<d_z$ be so that the degree of $q_{i},\,s_i<i\leq s_{i+1}$ is $d_i$, for any $i=0,1,...,z-1$. Then, as can bee seen from the proof of Proposition \ref{Prop Ord 2}, the sets $I_\ve(N)$ are not empty only when the permutation $\ve$ preserves the sets $\cD_i=\{s_i+1,...,s_{i+1}\}$, where $i=0,1,...,z-1$. In other words 
\begin{equation}\label{tenz}
\ve=\ve^{(0)}\otimes\ve^{(1)}\cdots\otimes\ve^{(z-1)}
\end{equation}
 is the tensor product  of $z$ permutations $\ve^{(i)}$ of the sets $\cD_i$, where $i=0,1,...,z-1$. We denote by $\cE_\ell'$ the set of all permutations with the above product structure.

Next, let $I_\ve(N),\,\ve\in\cE_\ell',$ be the sets from Proposition \ref{Prop Ord 2} and 
set 
\begin{eqnarray*}
G_N=\bigcup_{\ve\in\cE_\ell'}I_\ve(N),\,\, S_{1,N}=N^{-\frac12}\sum_{n\in G_N}F(\xi_{q_1(n,N)},\xi_{q_2(n,N)},...,\xi_{q_\ell(n,N)})\\\text{and }\, I_\ve(t,N)=I_\ve(N)\cap[1,Nt],\, t\in[0,1].
\end{eqnarray*}
Since the cardinality of the set $B_N=N\setminus G_N$ does not exceed $AN^r$ for some $A>0$ and $r\in(0,1)$, using Lemma \ref{VarLemma} we see that if the limits
\[
b(t,s)=\lim_{N\to\infty}\bbE[S_{1,[Nt]}S_{1,[Ns]}]
\]
exist, then we also have
$
b(t,s)=\lim_{N\to\infty}\bbE[\cS_N(t)\cS_N(t)].
$

For each $\ve\in\cE_\ell'$ set 
\begin{eqnarray*}
S_{1,\ve,N}=N^{-\frac12}
\sum_{n\in I_\ve(N)}F_\ve(\xi_{q_{\ve(1)}(n,N)},\xi_{q_{\ve(2)}(n,N)},...,\xi_{q_{\ve(\ell)}(n,N)})
\end{eqnarray*}
where 
$
F_\ve(x_{\ve(1)},x_{\ve(2)},...,x_{\ve(\ell)})=F(x_1,x_2,...,x_\ell).
$
Then $S_{1,N}=\sum_{\ve\in\cE_\ell'}S_{1,\ve,N}$.
For each $\ve$, we consider the decomposition of $F_\ve$ 
\[
F_\ve(x_{\ve(1)},...,x_{\ve(\ell)})=\sum_{j=1}^{\ell}F_{\ve,j}(x_{\ve(1)},...,x_{\ve(j)})
\]
where 
\[
F_{\ve,\ell}(x_1,...,x_\ell)=F_\ve(x_1,...,x_\ell)-\int F(x_1,...,x_\ell)d\mu(x_{\ve(\ell)})
\]
and for all $j=\ell-1,\ell-2,...,1$,
\begin{eqnarray*}
F_{\ve,j}(x_{\ve(1)},...,x_{\ve(j)})=\int F(x_1,...,x_\ell)d\mu(x_{\ve(\ell)})d\mu(x_{\ve(\ell-1)})\cdots 
d\mu(x_{\ve(j+1)})\\-\int F(x_1,...,x_\ell)d\mu(x_{\ve(\ell)})d\mu(x_{\ve(\ell-1)})\cdots 
d\mu(x_{\ve(j)}).
\end{eqnarray*}
Observe that for any $\ve$, $i$ and $y_{\ve(1)},...,y_{\ve(i-1)}$ we have
\begin{equation}\label{F_i's props}
\int F_{\ve,i}(y_{\ve(1)},...,y_{\ve(i-1)},x)d\mu(x)=0.
\end{equation}
Set 
\[
S_{\ve,j,N}=\sum_{n\in I_\ve(N)}F_{\ve,j}(\xi_{q_{\ve(1)}(n,N)},...,\xi_{q_{\ve(j)}(n,N)}).
\]
Then the limit $b(t,s)$ exists if  the limits 
\[
D_{\ve,\tau,i,j}(t,s)=\lim_{N\to\infty} \bbE[S_{\ve,i,[Nt]}S_{\tau,j,[Ns]}]
\]
exist, for any $\sig,\tau\in\cE_\ell'$ and $1\leq i,j\leq\ell$.

\subsection{Existence of the limiting covariances: proof of theorem \ref{VarThm}}\label{Sec Var proof}
In the course of the proof of Theorem \ref{VarThm} we will need the following
\begin{lemma}\label{b lemma}
For any $N$, $\ve,\tau\in\cE_\ell$, $1\leq i,j\leq\ell$, $n\in I_\ve(N)$ and $m\in I_\tau(N)$ set
\[
b_{\ve,\tau,i,j,N}(n,m)=\bbE[F_{\ve,i}(\xi_{q_{\ve(1)}(n,N)},...,\xi_{q_{\ve(i)}(n,N)})F_{\tau,j}(\xi_{q_{\tau(1)}(m,N)},...,\xi_{q_{\tau(j)}(m,N)})].
\]
Then 
\[
|b_{\ve,\tau,i,j,N}(n,m)|\leq C\upsilon_N\big(|q_{\ve(i)}(n,N)-q_{\tau(j)}(m,N)|\big)
\]
where $\upsilon_N(l)=\Big(\phi\big(\min([N^{\frac12}],[l/3])\big)\Big)^{1-\frac 2w}+\Big(\beta_q\big(\min([N^\frac12],[l/3])\big)\Big)^\kappa$ and $C$ is some constant.
\end{lemma}

Suppose that $\te\big(1-\frac 2w\big)>4$ (as in Theorems \ref{VarThm} and \ref{CLTthm}). Then, 
under Assumption \ref{Mix2} we have 
\[
\upsilon_N(l)\leq C\max(N^{-2},l^{-4})
\]
where $C$ is some constant. In particular 
\begin{equation}\label{gamma N lim}
\lim_{N\to\infty}N\upsilon_N(N)=0
\end{equation}
and 
\begin{equation}\label{Gamma}
\Upsilon:=\sup_{N}\sum_{l=0}^\infty (l+1)\upsilon_N(l)<\infty.
\end{equation}

\begin{proof}[Proof of Lemma \ref{b lemma}]
Relying on (\ref{B bound}),
the statement of the lemma clearly holds true when  $q_{\ve(i)}(n,N)=q_{\tau(j)}(m,N)$. Suppose now that
 $l=q_{\ve(i)}(n,N)-q_{\tau(j)}(m,N)>0$. Consider the variables $u_\ve=(x_{\ve(1)},...,x_{\ve(i)})$, $u_\tau=(y_{\tau(1)},...,y_{\tau(j)})$  and 
$u=(u_\ve,u_\tau)$. Consider also the function $H=H(u)$ given by 
\[
H(u)=F_{\tau,j}(u_\tau)F_{\ve,i}(u_\ve).
\]
Set 
\[
\cA_1=\{q_{\tau(w)}(m,N):\,1\leq w\leq j\}\cup\{q_{\ve(w)}(n,N):\,1\leq w< i\}
\,\text{ and }\,\cA_2=\{q_{\ve(i)}(n,N)\}.
\]
Then by (\ref{Different degree ordering}), (\ref{Same degree ordering}) for any $a\in \cC_1$ and $b\in\cC_2$ we have
\[
a\leq b-\min(l,c_1 N^{\frac12})
\]
where $c_1=\min(c,1)$. Consider the random vectors $U_1$ and $U_2$ given by 
\[
U_i=\{\xi_d:\,d\in\cA_i\}.
\]
Then 
\[
b_{\ve,\tau,i,j,N}(n,m)=\bbE H(U_1,U_2).
\]
Let $\tilde U_2$ be a copy of $U_2$ which is independent of $U_1$. Then by (\ref{F_i's props}),
\begin{eqnarray*}
\bbE H(U_1,\tilde U_2)\\=\bbE\big[F_{\tau,j}(\xi_{q_{\tau(1)}(m,N)},...,\xi_{q_{\tau(j)}(m,N)})\cdot
\bbE[F_{\tau,j}(\xi_{q_{\ve(1)}(n,N)},...,\xi_{q_{\ve(i-1)}(n,N)},U_2)|U_1]\big]=0
\end{eqnarray*}
where we used that the law of $U_2$ is $\mu$. Taking into account Assumption \ref{Moment Ass}, (\ref{F Hold}) and (\ref{F Bound}), applying Proposition \ref{1.3.14} we obtain that
\[
|b_{\ve,\tau,i,j,N}(n,m)|=|\bbE H(U_1,U_2)-\bbE H(U_1,\tilde U_2)|\leq 
 C\upsilon_N(l)
\]
where $C$ is some constant and $\upsilon_N(l)$ was defined in the statement of the lemma. The proof of the lemma in the case when $q_{\ve(i)}(n,N)-q_{\tau(j)}(m,N)<0$ is analogous.
\end{proof}

Now we will prove
\begin{proposition}\label{Var prop}
Suppose that the conditions of Theorem \ref{VarThm} hold true. Then the limits $b(t,s)$ exist. In fact:

(i) If $\deg q_{\ve(i)}\not=\deg q_{\tau(j)}$ then the limits $D_{\ve,\tau,i,j}(t,s)$ exist an equals $0$.

(ii) If $\ve=\tau$, $i=j$ and $\deg q_{\ve(i)}>1$ then
\[
D_{\ve,\tau,i,j}(t,s)=\min(s,t)c\int F_{\ve,i}^2(x_{\ve(1)},...,x_{\ve(i)})d\mu(x_{\ve(1)})d\mu(x_{\ve(2)})\cdots d\mu(x_{\ve(i)})
\]
where 
\[
c=\lim_{N\to\infty}N^{-1}|I_\ve(N)|=\sum_{j=1}^{l_\ve}\max\big(b_{j,\ve}-a_{j,\ve},0\big)
\]
and $b_{j,\ve}$ and $a_{j,\ve}$ are defined in Proposition \ref{Prop Ord 2}. 

(iii) If $q_{\ve(i)}$  and $q_{\tau(j)}$ have exploding differences then the limit $D_{\ve,\tau,i,j}(t,s)$ exists and equals $0$. 

(iv) If $q_{\ve(i)}$ and $q_{\tau(j)}$ are linearly related, $\deg q_{\ve(i)}=\deg q_{\tau(j)}=k>1$ and the differences of $q_{\ve(i)}$ and $q_{\tau(j)}$ do not explode,
 then the limit $D_{\ve,\tau,i,j}(t,s)$ exists and has the form
\[
c_{\ve,i,\tau,j}(s,t) \int F_{\ve,i}(\bar x)F_{\tau, j}(\bar y)d\cM(\bar x,\bar y)
\]
where $\cM$ is define by (\ref{cM def}) and the number $c_{\ve,i,\tau,j}(s,t)$ can be recovered from the proof.
\end{proposition}

\begin{proof}
As we have explained before, the limits $b(t,s)$ exist if  the limits 
\[
D_{\ve,\tau,i,j}(t,s)=\lim_{N\to\infty} \bbE[S_{\ve,i,[Nt]}S_{\tau,j,[Ns]}]
\]
exist, for any $\sig,\tau\in\cE_\ell'$ and $1\leq i,j\leq\ell$. 

Let $\ve,\tau\in\cE_\ell'$ and $1\leq i,j\leq\ell$. 
 When $\deg q_{\ve(i)}=\deg q_{\tau(j)}=1$, then existence of the limits $D_{\ve,\tau,i,j}(t,s)$ is obtained exactly as in Chapter 3 of \cite{book}, taking into account that $\ve,\tau$  have the tensor product structure (\ref{tenz}). For readers' convenince, let us explain the main idea of the proof in \cite{book}. 
First, it enough to show that the limit
\begin{equation}\label{3.3.14}
\lim_{N\to\infty}\frac 1N\sum_{\substack
{a_\ve N<m<b_\ve N,\,a_{\tilde\ve}N<n<b_{\tilde\ve}N\\
\rho_{\ve_i}(m,N)-\rho_{\tilde\ve_j}(n,N)=u}}
b_{i,j,\ve,\tilde\ve}(N,m,n)
\end{equation}
exists for each integer $u$ (where the sum over the empty set is considered
to be zero) and since then we can just over all $u$'s. 
The idea behind proving (\ref{3.3.14}) is to show that  that if $N\to\infty$, $m-a_\ve N\to\infty,\, b_\ve N-m
\to\infty,\, n-a_{\tilde\ve}N\to\infty,\, b_{\tilde\ve}N-n\to\infty$ and
$\rho_{\ve_i}(m,N)-\rho_{\tilde\ve_j}(n,N)=p_{\ve_i}m-p_{\tilde\ve_j}n+
N(q_{\ve_i}-q_{\tilde\ve_j})=u$ then $b_{i,j,\ve,\tilde\ve}(N,m,n)$
converges to a limit. After that, using Assumption \ref{A1} we estimated the number of solutions 
of the equation
\begin{equation}\label{3.3.15}
p_{\ve_i}m-p_{\tilde\ve_j}n+N(q_{\ve_i}-q_{\tilde\ve_j})=u
\end{equation}
in $m\in\cN_{\ve,N}$ and $n\in\cN_{\tilde\ve,N}$, divided this number by
$N$ and showed that this ratio converges to a limit which will give the
limit in (\ref{3.3.14}). We refer the readers' to the Section 3.3.2 in Chapter 3 of \cite{book} for the precise details.

Next, suppose that $\max\big(\deg q_{\ve(i)},\deg q_{\tau(j)}\big)>1$.
In what follows, we will always assume that $\deg q_{\ve(i)}\geq \deg q_{\tau(j)}$. We first write
\begin{equation}\label{First write}
\bbE[S_{\ve,i,[Nt]}S_{\tau,j,[Ns]}]=N^{-1}\sum_{n\in I_\ve(N,t)}\sum_{m\in I_\tau(N,t)}
b_{\ve,\tau,i,j,N}(n,m)
\end{equation}
where $I_\ve(N,t)=I_\ve(N)\cap[1,Nt]$ and $I_\tau(N,s)=I_\tau(N)\cap[1,Ns]$.
Suppose first that  $\deg q_{\ve(i)}>\deg q_{\tau(j)}$. Then by (\ref{Different degree ordering}) we have 
\[
\min_{N^{r_0}\leq n,m\leq N}\big(q_{\ve(i)}(n,N)-q_{\tau(j)}(m,N)\big)\geq N 
\]
and therefore, by Lemma \ref{b lemma}, for any $n\in I_\ve(N)$ and $m\in I_\tau(N)$,
\begin{eqnarray}
|b_{\ve,\tau,i,j,N}(n,m)|\leq C\upsilon_N(N).
\end{eqnarray}
Hence, by (\ref{gamma N lim}) we have
\[
\left|\bbE[S_{\ve,i,[Nt]}S_{\tau,j,[Ns]}]\right|\leq N\upsilon_N(N)\to0\,\,\text{ when }\,N\to\infty.
\]

Now we will consider the case when $\deg q_{\ve(i)}=\deg q_{\tau(j)}>1$.
Suppose that $\ve=\tau$ and $i=j$.
If $n\not=m$ then 
\[
|q_{\ve(i)}(n,N)-q_{\ve(i)}(m,N)|\geq C_1\min(n,m)
\]
where $C_1>0$ is some constant,
and therefore by Lemma \ref{b lemma} we have
\[
|b_{\ve,\tau,i,j,N}(n,m)|=|b_{\ve,\ve,i,i,N}(n,m)|\leq C\upsilon_N(C_1\min(n,m))
\]
which implies that 
\[
\lim_{N\to\infty}N^{-1}\sum_{n\in I_\ve(N,t)}\sum_{m\in I_\tau(N,t),\,m\not=n}
|b_{\ve,\ve,i,i,N}(n,m)|=0
\]
where we used (\ref{gamma N lim}). On the other hand, when $n=m$ we have
\[
b_{\ve,\ve,i,i,N}(n,n)=\int F^2_{\ve,i}(x_{\ve(1)},...,x_{\ve(i)})d\mu(x_{\ve(1)})\cdots d\mu(x_{\ve(i)})
\]
and therefore by (\ref{First write}),
\[
\lim_{N\to\infty} \bbE[S_{\ve,i,[Nt]}S_{\ve,j,[Ns]}]=\int F^2_{\ve,i}(x_{\ve(1)},...,x_{\ve(i)})d\mu(x_{\ve(1)})\cdots d\mu(x_{\ve(i)})
\]
and the proof of Proposition \ref{Var prop} (ii) is completed.

Suppose next that the differences of $q_{\ve(i)}$  and $q_{\tau(j)}$ explode. Fix some  $\del\in(0,1)$. Then there exist constants  $C_\del>0$, $N_\del>0$ and sets
$\Gam_{N,\del}\subset[1,N]$  of integers whose cardinality does not exceed $\del N$ so that
 for any $N>N_\del$ and $n\in[\del N,N]\setminus\Gam_{N,\del}$, 
\[
\min_{m\in[N\del,N]}|q_{\ve(i)}(n,N)-q_{\tau(j)}(m,N)|\geq C_\del N.
\]
Write 
\[
\bbE[S_{\ve,i,[Nt]}S_{\tau,j,[Ns]}]=J_1+J_2+J_3
\]
where $J_i=J_i(N,\ve,\tau,i,j,s,t,\del)$ are given by
\begin{eqnarray*}
J_1=N^{-1}\sum_{n\in I_\ve(N,t)\setminus \Gam_{N,\del}}\,\sum_{m\in I_\tau(N,s)\cap[\del N, N]}
b_{\ve,\tau,i,j,N}(n,m),\\
J_2=N^{-1}\sum_{n\in I_\ve(N,t)\setminus \Gam_{N,\del}}\,\sum_{m\in I_\tau(N,s)\cap[1,\del N)}
b_{\ve,\tau,i,j,N}(n,m)\,\,\text{ and }\\
J_3=N^{-1}\sum_{n\in I_\ve(N,t)\cap \Gam_{N,\del}}\,\sum_{m\in I_\tau(N,s)}
b_{\ve,\tau,i,j,N}(n,m).
\end{eqnarray*}
We will show that the upper limit as $N\to\infty$ of each one of the $J_i$'s does not exceed $\del$, which by taking $\del\to0$ will complete the proof of Proposition \ref{Var prop} (iii).
First by Lemma \ref{b lemma},
\[
|J_1|\leq N^{-1}\sum_{n\in I_\ve(N,t)\setminus \Gam_{N,\del}}\sum_{m\in I_\tau(N,s)\cap[N\del, N]}
|b_{\ve,\tau,i,j,N}(n,m)|\leq C'_\del N\upsilon_N([C_\del N])
\]
for some constant $C'_\del>0$,
and so by (\ref{gamma N lim}) we have $\lim_{N\to\infty}J_1=0$. Next,
\[
|J_2|\leq N^{-1}\sum_{m=1}^{\del N}\sum_{k=0}^\infty\sum_{n\in A_k(m,N)}|b_{\ve,\tau,i,j,N}(n,m)|
\]
where 
$
A_k(m,N)
$ 
is the set of all $n$'s so that $|q_{\ve(i)}(n,N)-q_{\tau(j)}(m,N)|=k$. Since the $q_i$'s are polynomials, $|A_k(m,N)|\leq C_2$ for some constant $C_2$ which does not depend on $m$ and $k$.
It follows from Lemma \ref{b lemma} that  
\[
|J_2|\leq C_2N^{-1}\sum_{m=1}^{\del N}\sum_{k=0}^\infty\upsilon_N(k)\leq C_2\Upsilon\del
\]
where $\Upsilon$ was defined in (\ref{Gamma}). Since the cardinality of the set $\Gam_{N,\del}$ does not exceed $\del$, similar arguments show that for any sufficiently large $N$,
\[
|J_3|\leq A\del
\]
for some constant $A$.

Next, suppose that $q_{\ve(i)}$  and $q_{\tau(j)}$ are linearly related and that their differences do not explode. Then, by Corollary \ref{Cor linear case} they are equivalent, namely there exists rational $c$ and $r$ so that 
\[
q_{\ve(i)}(n,N)=q_{\tau(j)}(cn+r,N)
\]
for any natural $n$. If $m\not=cn+r$ then by Corollary (\ref{Cor linear case}) we have
\[
|q_N(m)-p_N(n)|\geq Q(cn/N)N^{k-1}C_0
\]
for some continuous function $Q$ which is strictly positive on $(0,1]$, 
 and so by Lemma \ref{b lemma},
\[
|b_{\ve,\tau,i,j,N}(n,m)|\leq C\upsilon_N\big(Q(cn/N)N^{k-1}C_0\big).
\]
Therefore, taking into account (\ref{gamma N lim}), in order to prove that the limit $D_{\ve,\tau,i,j}(t,s)$ exists, it is enough to show that the limit
\begin{equation}\label{Sum over}
\lim_{N\to\infty}N^{-1}\sum_{n\in I_\ve(N,t)}b_{\ve,\tau,i,j,N}(n,cn+r)\bbI(cn+r\in I_\tau(N,s)\cap \bbN)
\end{equation}
exists, where  $x\to I(x\in A)$ is the indicator function of a set $A$.
 Since the sets $I_{\ve}(N)$ and $I_{\tau}(N)$ are unions of intervals whose length is proportional to $N$ (i.e they have the form (\ref{I form})), and the set 
\[
\{n\in\bbN:\,cn+r\in\bbN\}
\]
is a finite union of arithmetic progressions, in order to show that the above limit exists it is enough to show that there exist sets $\Del_{N,\del}$ whose cardinalities do not exceed $\del N$ so that for any $\del\in(0,1)$  the limit
\begin{equation}\label{eps tau limit}
\lim_{N\to\infty,\,n\in I_\ve(N)\setminus\Del_{N,\del},\,
cn+r\in I_\tau(N)}b_{\ve,\tau,i,j,N}(n,cn+r)
\end{equation}
exists, and that this limit does not depend on $\del$. Fix some $\del\in(0,1)$. Suppose that $q_{\ve(u)}(x,N)=a_{\ve(u)}x+b_{\ve(u)} N$ and $q_{\tau(v)}(x,N)=a_{\tau(v)}x+b_{\tau(v)}N$ are  linear polynomials  where $u<i$ and $v<j$. Then 
\[
q_{\ve(u)}(n,N)-q_{\tau(v)}(cn+r,N)=a_{\ve(u)} r+(a_{\ve(u)} c-a_{\tau(v)})n+(b_{\ve(u)}-b_{\tau(v)})N.
\] 
If $ca_{\ve(u)}\not=a_{\tau(v)}$ but $b_{\ve(u)}=b_{\tau(v)}$ then 
for any $n\in[\del N,N]$ we have 
\[
|q_{\ve(u)}(n,N)-q_{\tau(v)}(cn+r,N)|\geq C_\del N
\]
for some constant $C_\del>0$. When  $ca_{\ve(u)}\not=a_{\tau(v)}$ and $b_{\ve(u)}\not=b_{\tau(v)}$, there exists a constant $w=w_{u,v,N}\in\bbR$ so that for any sufficiently large $N$ and $n\in[\del N,N]\setminus[N(w-\del),N(w+\del)]:=J(w,\del, N)$,
\[
|q_{\ve(u)}(n,N)-q_{\tau(v)}(cn+r,N)|\geq C'_\del N
\]
where $C'_\del>0$ is some constant. Let $(u_j,v_j),\,j=1,...,L_1$ be the indexes so that $q_{\ve(u_j)}$ and $q_{\tau(v_j)}$ are linear,  $ca_{\ve(u_j)}\not=a_{\tau(v_j)}$ and $b_{\ve(u_j)}\not=b_{\tau(v_j)}$. Let $q_{\ve(1)},...,q_{\ve(U)}$ and $q_{\tau(1)},...,q_{\tau(V)}$ be the linear polynomials among the $q_{\ve(u)}$'s and the $q_{\tau(v)}$'s, respectively, where $1\leq u<i$ and $1\leq v<j$. We conclude that for any
\[
 n\in \bigcup_{j=1}^{L_1}J(w_{u_j,v_j,N},\del,N):=\Theta_{N,\del}
\]
we can order the numbers $q_{\ve(u)}(n,N)$ and $q_{\tau(v)}(cn+r,N)$, where $1\leq u\leq U$ and $1\leq v\leq V$, so that the differences between each consecutive numbers in this ordering is either a constant which does not depend on $n$ and $N$, or it is not less than  $B_\del N^{\frac12}$  for some $B_\del>0$, where we also used (\ref{Same degree ordering}). 
Combining this with Assumption \ref{A2}, we deduce from Corollary \ref{Cor linear case} that the terms
\[
q_{\ve(l)}(n,N)-q_{\tau(z)}(cn+r,N)
\]
where $1\leq l\leq i$ and $1\leq z\leq j$,
 either have aboslute values bounded from below by some $E_\del N$, $E_\del>0$ or they are constants. Set $\Del_{N,\del}=[\del N,N]\setminus\Theta_{N,\del}$. Using (\ref{Different degree ordering}) and (\ref{Same degree ordering}), applying  Proposition \ref{1.3.14} we derive that there exist constants $d_1,...,d_{L}$ ($L\geq0$) and indexes $(l_e,z_e),\,e=1,2,...,L$ so that 
\[
\lim_{N\to\infty,\,n\in I_\ve(N)\setminus\Del_{N,\del},\,
cn+r\in I_\tau(N)}b_{\ve,\tau,i,j,N}(n,cn+r)
=\int F_{\ve,i}(\bar x)F_{\tau, j}(\bar y)d\cM
\]
where $\bar x=(x_{\ve_1},...,x_{\ve(i)})$, $\bar y=(y_{\tau(1)},...,y_{\tau(j)})$ and
$\cM$ has the form
\begin{equation}\label{cM def}
d\cM(\bar x, \bar y)=\prod_{l\not\in\{l_e\}}d\mu(x_{\ve(l)})\prod_{z\not\in\{z_e\}}d\mu(y_{\tau(z)})
\prod_{e=1}^L d\mu_{d_e}(x_{\ve(l_e)},y_{\tau(z_e)}).
\end{equation}
The indexes $(l_e,z_e)$ are exactly the ones for which 
\[
q_{\ve(l_e)}(n,N)-q_{\tau(z_e)}(cn+r,N)=d_e=q_{\ve(l_e)}(r,0)-q_{\tau(z_e)}(0,0)
\]
where $d_e$ does not depend on $n$ and $N$.
\end{proof}

\subsection{Existence under abstract number theory type conditions}\label{General Cov}
Let $\ve,\tau\in\cE'_\ell$ and $1\leq i,j\leq\ell$ be so that $\deg q_{\ve(i)}=\deg_{\tau(j)}=k>1$. 
Then
\[
 \bbE[S_{\ve,i,[Nt]}S_{\tau,j,[Ns]}]=\frac{1}N\sum_{n\in I_\ve(N,t)}\sum_{v\in\bbZ}\,\sum_{m\in J_\ve(N,s,v,i,\ve)}b_{\ve,\tau,i,j,N}(n,m)
\] 
where $J_\ve(N,s,v,u,\ve)$ is the set of all $m$'s in $I_\tau(N,s)$ so that $q_{\tau(j)}(m,N)-q_{\ve(i)}(n,N)=v$.
Taking into account Lemma \ref{b lemma}, if the limit
\[
\Gam_{v}=\lim_{N\to\infty}\frac{1}N\sum_{n\in I_\ve(N,t)}\,\sum_{m\in J_\ve(N,s,v,i,\ve)}b_{\ve,\tau,i,j,N}(n,m)
\]
exists for each $v$, then 
\[
D_{\ve,\tau,i,j}(t,s)=\lim_{N\to\infty}\bbE[S_{\ve,i,[Nt]}S_{\tau,j,[Ns]}]=\sum_{v\in\bbZ}\Gam_{v}.
\]
Using Lemma \ref{Lemma} with $q=q_{\tau(j)}$ and $p=q_{\ve(i)}$, the equality  $q_{\tau(j)}(m,N)-q_{\ve(i)}(n,N)=v$ means that, with $y=n/N$, 
\[
H_{N,y}(m-N\gam_k(y)-R_k(y))=\frac{v+P_0+Q(R_k(y),0)}{N^{k-1}Q_k'(\gam_k(y))}.
\]
Fix some $\del\in(0,1)$. 
When $|u|\leq \sqrt N$ and $y\geq\del$ then, since $H_{N,y}$ is one to one on intervals of the form $[-a,a],\,a>0$ (when $N$ is large enough), we obtain that for any sufficiently large $N$ there exists a unique solution  $x=x_{n,N,v}$ to the equation 
\[
H_{N,y}(x)=\frac{v+P_0+Q(R_k(y),0)}{N^{k-1}Q_k'(\gam_k(y))}.
\]
Using Lemma \ref{b lemma} the following proposition follows similarly to the proof of Proposition \ref{Var prop}.
\begin{proposition}
The limit $D_{\ve,\tau,i,j}(t,s)$
exists if for any $v$ and sufficiently large $n$ and $m$ so that  $q_{\tau(j)}(m,N)-q_{\ve(i)}(n,N)=v$ the differences 
\[
q_{\tau(j')}(m,N)-q_{\ve(i')}(n,N),\,i'\leq i,\,\,j'\leq j
\]
are either constants or they converge to $\infty$ (in absolute value) as $n\to\infty$ and the limit 
\[
\lim_{N\to\infty}\frac 1N\left|\{n\in I_{\ve}(N,t):\,x_{n,N,v}+\gam_k(n/N)+R_k(n/N)\in\bbN\}\right|
\] 
exists, where $|\Gamma|$ stands for the cardinality of a finite set $\Gamma$.
\end{proposition}
Note that when the above conditions hold true only with $s=t=1$ then we obtain that the limit $D^2$ exists, which is enough in order to derive that $\cS_N(1)$ converges in distribution towards a centered normal random variable whose variance is $D^2$.





\subsection{Positivity of the asymptotic variance}\label{PosSec}
We assume here  that the conditions of Theorem \ref{VarThm} are satisfied. Set $S_N=\cS_N(1)$ and
\[
D^2=\lim_{N\to\infty}\bbE S_N^2=\lim_{N\to\infty}\text{Var}(S_N).
\]

We will say that the polynomials $q_i(x,N)$ and $q_j(x,N)$ are equivalent if there exist rational $c$ and $r$ so that  $q_i(n,N)-q_j(cn+r,N)$ is a constant function of $n$ and $N$  (in Remark \ref{Remark equiv} we called this equivalence relation $\bbQ$-equivalence). Then any two equivalent polynomials have the same degree. Let $q_{s+1},...,q_{\ell}$  be the non-linear polynomials among the $q_i$'s and consider the decomposition of $\{q_i:\,i>s\}$ into  equivalence classes $A_1,A_2,...,A_w$, ordered so that the degree of each member of $A_i$ does not exceed the degree of each member of $A_{i+1}$, $i=1,...,w-1$. For any $\ve,\tau,i,j$ so that $q_{\ve(i)}$ and $q_{\tau(j)}$ are not both linear and not equivalent we have 
\[
D_{\ve,\tau,i,j}:=D_{\ve,\tau,i,j}(1,1)=0.
\]
Therefore,
$
D^2=\sum_{t=0}^{w}D_t^2
$
where for $t=1,2,...,w$,
\[
D_t^2=\sum_{\ve,\tau}\sum_{i,j: q_{\ve(i)},q_{\tau(j)}\in A_t}D_{\ve,\tau,i,j}.
\]
and $D_0^2:=D^2-\sum_{t=1}^wD_t^2$ (we will see soon that it also nonnegative). Hence, $D^2=0$ if and only if $D_t^2=0$ for any $t$.

When $q_1$ is linear and $q_1,...,q_{s}$ are the linear polynomials among the $q_i$'s, we define the function $G$ by
\[
G(x_1,...,x_s)=\int F(x_1,...,x_s,y_{s+1},...,y_\ell)d\mu(y_{s+1})\cdots d\mu(y_\ell).
\]
Then, taking into account the tensor product structure of the set of permutations $\cE_\ell'$ from (\ref{tenz}) we have 
\begin{eqnarray*}
\sum_{n=1}^N G(\xi_{q_1(n,N)},...,\xi_{q_s(n,N)})=\sum_{\ve\in\cE_\ell'}\sum_{n\in I_\ve(N)}\sum_{i=1}^sF_{\ve,i}(\xi_{q_{\ve(1)}(n,N)},...,\xi_{q_{\ve(i)}(n,N)})\\
+\sum_{n\in B_N} G(\xi_{q_1(n,N)},...,\xi_{q_s(n,N)}):=G_{N,1}+G_{N,2}
\end{eqnarray*}
where $B_N$ is the set from Proposition \ref{Prop Ord 2}. By Lemma \ref{VarLemma}, we have $\lim_{N\to\infty}N^{-1}\bbE G_{N,2}^2=0$ (since the cardinality of $B_N$ grows sublinearly in $N$) and hence
 \[
D_0^2=\lim_{N\to\infty}N^{-1}\bbE G_N^2
\]
where 
\[
G_N=\sum_{n=1}^{N}G(\xi_{q_1(n,N)},...,\xi_{q_{s}(n,N)}).
\]
Recall next the following result from Chapter 3 in \cite{book}, which was stated there as Theorem 3.3.4, and is reformulated here for the function $G$:
\begin{theorem}\label{thm3.3.2} Let 
$\xi^{(i)}=\{\xi^{(i)}_n: n\geq 0\}$, $i=1,2,...,s$ 
be independent copies of the random sequence $\xi=\{\xi_n,\, m\geq 0\}$.
Suppose that the conditions of Theorem \ref{VarThm} hold true and
$\sig^2=\lim_{N\to\infty} N^{-1}\bbE(G_N)^2$ exists. Let $a_i,b_i$ be the numbers from (\ref{linear}). Set
\[
U_{n,N}=G(\xi^{(1)}_{a_1n+b_1N},\xi^{(2)}_{a_2n+b_2N},...,\xi^{(s)}_{a_\ell n
+b_\ell N})
\]
$\Sig_N=\sum_{1\leq n\leq N}U_{n,N}$
and 
\[
v^2=\lim_{N\to\infty}N^{-1}\bbE(\Sig_N)^2.
\]
Then, the limit $v^2$ exists, and $\sig^2>0$ if and only if $v^2>0$.
The latter condition holds true if and only if there exists no 
representation of the form
\begin{equation}\label{Cob0}
U_n=Z_{n+1}-Z_n,\,\, n=0,1,2,...
\end{equation}
where $U_n=G(\xi^{(1)}_{a_1 n},...,\xi^{(s)}_{a_{s}n})$, $\{ Z_n\}_{n=0}^\infty$ is a square integrable weakly (i.e. in 
the wide sense) stationary process. Furthermore, $v^2=0$ if and only
if Var$\Sig_N$ is bounded.
\end{theorem}

For readers' convenience, let us give some of the ideas behind the proof of Theorem \ref{thm3.3.2}.
Let us first explain (\ref{Cob0}).
 Since the above copies are independent and the distribution of $(\xi_{n},\xi_m)$ depends only on $n-m$
 we have 
\[
\bbE\Sig_N^2=\bbE\big(\sum_{n=1}^N U_n)^2.
\] 
Observe next that the seqeunce $\big\{G(\xi^{(1)}_{a_1 n},...,\xi^{(s)}_{a_{s}n})\big\}$ is centered and stationary in the wide sense. We recall now the following result, which is a combination of  Proposition 8.3 and Theorem 8.6 of \cite{Br} (modified for a one sided sequence): let  $\{U_n\}$ be a centered sequence of weakly stationary random variables  with correlations $b_n=\text{Cor}(U_n,U_0)$ satisfying $\sum_{n=0}^{\infty}(n+1)|b_n|<\infty$. Then  the limit $v^2=\lim_{n\to\infty}\frac1n\bbE(\sum_{j=0}^{n-1}U_n)^2$ exists, and it is positive if and only if there exists a weakly stationary process $Z_n$ so that $U_n=Z_{n+1}-Z_n$ for any $n\geq0$.
We conclude that $\sig^2=0$ if and only there exists a stationary in the wide sense sequence $\{Z_n\}$ of random variables so that for any $n$,
\begin{equation}\label{Coboundary}
G(\xi^{(1)}_{a_1 n},...,\xi^{(s)}_{a_{s}n})=Z_{n+1}-Z_n.
\end{equation}
We note that the correlations of $\{U_n\}$ indeed satisfy the above conditions because of the definition of $G$, (\ref{F Hold}), (\ref{F bar}) and  Assumption \ref{Mix2}. We also remark that $U_n$ is indeed centered because of our assumption that $\bar F=0$.

Next, for readers' convenience we will give the idea behind the equivalence $\sig^2=0\iff v^2=0$ in the case when $s=2$,  namely when there are only two linear polynomials $q_1(n,N)=a_1n+b_1N$ and $q_2(n,N)=a_2n+b_2N$. In the case when $b_1=b_2$ we have
\[
\mathbb E[G(\xi_{q_1(n,N)},\xi_{q_1(n,N)})G(\xi_{q_1(m,N)},\xi_{q_2(m,N)})]=
E[G(\xi_{a_1n},\xi_{a_2n})G(\xi_{a_1m},\xi_{a_2m})]
\]
and therefore the whole problem is reduced to the case when $q_1$ and $q_2$ do not depend on $N$, and this case was covered in Section 4 of \cite{HK2}. We assume without a loss of generality that $b_2>b_1$. The proof in the case when $a_1\leq a_2$  proceeds exactly as in \cite{HK2}. Indeed, we can use the same block partitions which were construted in \cite{HK2} in the case when $q_1(n)=(a_1+b_1)n$ and $q_2(n)=(a_2+b_2)n$. The distance between $(a_2+b_2)n$ and $(a_1+b_1)_2$ is smaller than the distance between $a_2n+b_2N$ and $a_1n+b_1N$, and so the proof in \cite{HK2} proceeds similarly.
We note that only the case when $q_i(n)=in$ was considered in \cite{HK2}, but the general case when all $q_i$'s have the form $q_i(n)=l_in$, $i=1,2,...,s$, $l_i\in\bbN$ follows from this case by considering functions $F(x_1,...,x_{l_{s}})$ which depend only on the variables $x_{l_i}$.

In view of the above, we will focus here on the case when $a_1>a_2$. We will show that, in fact, in this case we have $v^2=\sig^2$.
Consider the functions $G_1$ and $G_2$ given by
\[
G_1(x)=\int G(x,y)d\mu(y),\text{ and }\, G_2(x,y)-G_1(x).
\]
Then 
\begin{equation}\label{Conditional}
\int G_2(x,y)d\mu(y)=0 \text{ for any } x.
\end{equation}
Set $\alpha=\frac{b_2-b_1}{a_1-a_2}$. Then $\alpha$ satisfies that $q_1(n,N)\leq q_2(n,N)$ for any $1\leq n\leq \alpha N$, while $q_2(n,N)\leq q_1(n,N)$ for any $N\geq n\geq \alpha N$. Using the above notations we have 
\[
\sum_{n=1}^NG(\xi_{q_1(n,N)},\xi_{q_1(n,N)})=\sum_{1\leq i,j\leq 2}S_{i}^{(j)}(N)
\]
where 
\begin{eqnarray*}
S^{(1)}_1(N)=\sum_{n=1}^{\alpha N}G_1(\xi_{a_1n+b_1N}),\,
S^{(1)}_2(N)=\sum_{n=1}^{\alpha N}G_2(\xi_{a_1n+b_1N},\xi_{a_2n+b_2N})
\\
S^{(2)}_1(N)=\sum_{\alpha N<n\leq N}G_1(\xi_{a_2n+b_2N}),\,
S^{(2)}_2(N)=\sum_{\alpha N<n\leq N}G_2(\xi_{a_2n+b_2N},\xi_{a_1n+b_1N}).
\end{eqnarray*}
Using (\ref{Conditional}) we get that for $i=1,2$
\begin{equation}\label{C}
|\mathbb E[S_1^{(i)}(N)S_2^{(2)}(N)]|\leq C
\end{equation}
for some constant $C>0$. 
In the case when $\xi_n$'s were independent for any $n<\alpha N$ and $m>\alpha N$ we have
\begin{eqnarray}
\mathbb E[G_1(\xi_{a_1n+b_1N})G_2(\xi_{a_2m+b_2N},\xi_{a_1m+b_1N})]\\=
\mathbb E[G_1(\xi_{a_1n+b_1N})\mathbb E\big[G_2(\xi_{a_2m+b_2N},\xi_{a_1m+b_1N})|\xi_{a_1n+b_1N}]\big]=0
\end{eqnarray}
since the conditional expectation inside the expectation vanishes due to (\ref{Conditional}). This shows that we can take $C=0$ in the independent case when $i=1$, and similar idea yields the same with $i=2$. When $\xi_n$'s are not independent 
we can use our mixing and approximation conditions from Assumption \ref{Mix2} in order to derive (\ref{C}) (applying Lemma \ref{b lemma}).
Similar arguments yield that $|E[S^{(1)}_1(N)S^{(1)}_2(N)]|$ and $|E[S^{(1)}_2(N)S^{(2)}_2(N)]|$ are bounded in $N$.
We conclude that 
\[
\mathbb ES_N^2=c_N+\sum_{1\leq,j,\leq 2}\mathbb E\big(S_{i}^{(j)}\big)^2
\]
where $c_N$ is some bounded sequence. The above equality also holds true when we replace $\xi_{a_1n+b_1N}$ with 
$\tilde\xi_{a_1n+b_1N}$, where $\tilde\xi$ is an independent copy of $\xi$.  It is enough to show that $E\big(S_{2}^{(i)}\big)^2$ behaves the same as in the case when the left and right coordinates are independent. We will explain why this is true for $j=1$. Write 
\[
S_2^{(1)}(N)=A_N+B_N
\]  
where $A_N=\sum_{n=1}^{M_N}G_2(\xi_{a_1n+b_1N},\xi_{a_2n+b_2N})$ and $B_N=S_2^{(1)}(N)-A_N$. Here $M_N$ satisfies that 
$a_1M_N+b_1N\thickapprox(b_2-\ve)N$, where $\ve>0$ is sufficiently small (we can just take $M=\beta N$ for an appropriate $\beta<\alpha$). This means that the difference between the largest index appearing in the left coordinate of $G_2$ in the sum $A_N$ is much smaller than the minimal index appearing in the right coordinate. Therefore, 
\[
\mathbb EA_N^2=c_N+\mathbb E (\tilde A_N)^2
\]
where $c_N$ is some bounded sequence and $\tilde A_N$ is defined similarly to $A_N$ but with $\tilde\xi_{a_1n+b_1N}$ in place of
$\xi_{a_1n+b_1N}$.
Using the same reasoning which lead to (\ref{C}), we obtain that $|\mathbb E A_Nb_N|$ is bounded in $N$. Therefore,
\[
\mathbb E(S_2^{(1)}(N))^2=\mathbb E (\tilde A_N)^2+\mathbb EB_N^2.
\]
Now we decompose $B_N$ into two blocks. On the first block the maximal index in the left coordinate will is  smaller than the minimal index  in the right coordinate, and  the expectation of the product of the two blocks is bounded in $N$. Proceeding this way we derive that $\sig^2=v^2$.\qed

The next result we have is the following
\begin{theorem}\label{VarEqThm}
Suppose that any two non-linear polynomials $q_i$ and $q_j$ are not equivalent. Then $D^2=0$ if and only if (\ref{Coboundary}) holds true and for any $\ve$ and $i$ so that $\deg q_{\ve(i)}>0$ we have $F_{\ve,i}^2(x_1,...x_i)=0$ for $\mu^i=\mu\times\cdots\mu$ almost any $(x_1,...,x_i)$.
\end{theorem}
\begin{proof}
Let $\deg q_{\ve(i)}$ and $\deg q_{\tau(j)}$ be two non-linear polynomials. If they not equivalent then by Proposition \ref{Var prop} we have 
\[
D_{\ve,\tau,i,j}(s,t)=0
\]
for any $s$ and $t$. Suppose that they are equivalent. Then, using the assumptions of Theorem \ref{VarEqThm}, we must have 
$q_{\ve(i)}=q_{\tau(j)}$. We claim that in this case $\ve=\tau$.
 Indeed, since $q_{\ve(i)}=q_{\tau(j)}$ then we have $c=1$ in (\ref{eps tau limit}). The sets $I_\ve(N)$ and $I_\tau(N)$ are disjoint when $\ve\not\not=\tau$, and in this case the sum in (\ref{Sum over}) is over a set of $n$'s whose asymptotic density is $0$. Therefore, either $D_{\ve,\tau,i,j}(t,s)=0$ for any $t,s$ or $\ve=\tau$, $i=j$ and 
 \[
D_{\ve,\tau,i,j}(1,1)=D_{\ve,\ve,i,i}(1,1)=s_\ve\int F_{\ve,i}^2\mu^i.
 \]
We conclude that
 \[
 D^2=\sum_{\ve,i:\,\deg q_{\ve(i)}>1}s_\ve\int F_{\ve,i}^2\mu^i+D_0^2
 \]
 and the proof of the proposition is completed using Theorem \ref{thm3.3.2}.
\end{proof}

\begin{remark}
The proof of Theorem \ref{VarEqThm} shows that if there exists one non-linear polynomial $q_{i_0}$ which is not equivalent to any other of the $q_j$'s (i.e. there exists a non-linear equivalence class which is a singleton), then 
\[
D^2=D_0^2+\sum_{\ve}s_{\ve}\int F_{\ve,\ve^{-1}(i_0)}^2(x_1,...,x_{\ve^{-1}(i_0)})d\mu(x_1)\cdots d\mu(x_{\ve^{-1}(i_0)})+
\del^2
\]
where $\del^2>0$ is the part of the variance which comes from the non-linear polynomials $q_{\ve(i)}$ and $q_{\tau(j)}$ which do not equal $q_{i_0}$. Therefore, we derive that $D^2>0$ if one of the functions $F_{\ve,\ve^{-1}(i_0)}$ is not identically $0$ \,($\mu^{\ve^{-1}(i_0)}$-almost surely). Moreover, when there exists one class of equivalent non-linear polynomials for which $c=1$ in (\ref{eps tau limit}) for any $q_{\ve(i)},q_{\tau(j)}$ in this class, then the proof of Theorem \ref{VarEqThm} shows that $D^2>0$ if  for any $\ve$ the functions $\{u_{i}^{(\ve)}\}\to s_\ve \sum_{i: q_{\ve(i)}\in A}F_{\ve,i}(u_{i}^{(\ve)})$ vanish with respect to an appropriate probability measure $\ka_{A,\ve}$, where $A$ is  the underlying class and $\{u_{i}^{(\ve)}\}$ is a set of variables which depends on $\ve$.
 Indeed, when $q_{\ve(i)}$ and $q_{\tau(j)}$ are equivalent and non-linear and $c=1$ in (\ref{eps tau limit}) we get that 
\[
D_{\ve,\tau,i,j}=\mathbb I(\tau=\ve)s_\ve\int F_{\ve,i}F_{\ve,j}d\cM_{\ve,i,j,A}
\]
for some probability measure $\cM_{\ve,i,j,A}$ (which was constructed in the proof of Proposition \ref{Var prop}). It
remains to show that the measures $\cM_{\ve,i,j}$ are an appropriate marginal of a measure $\cM_{\ve,A}$ which does not depend on $i$ and $j$ (this is done as in Lemma 7.1 in \cite{HK3}, where we recall that $\ve$ is fixed). We conclude that if any two equivalent non-linear polynomials $q_i$ and $q_i$ satisfy $q_i(x,y)=q_j(x+r,y)+z$ for some $r,z\in\bbZ$, then $D^2=0$ if and only if (\ref{Coboundary}) holds true and the  functions $\sum_{i: q_{\ve(i)}\in A}F_{\ve,i}(u_{i}^{(\ve)})$ vanish with respect to $\cM_{\ve,A}$ for any $\ve$ and a non-linear class $A$.
\end{remark}

\begin{remark}
When the polynomials $q_i$ are ordered so that $q_1(n,N)\leq q_2(n,N)\leq...\leq q_\ell(n,N)$ for any sufficiently large $n$ and $N$, then Theroem 2.3 in \cite{HK3} is proved exactly in the same why. This means that there exists a family of functions $G_{\al}$ and a family of probability measures $\ka_\al$ so that 
 $D^2=0$ if and only if  (\ref{Coboundary}) holds true and $G_\al^2=0$,\,$\ka_\al$ almost surely, for any $\al$.
 The measures and the functions can be constructed explicitly, see Section 7 in \cite{HK3}.
\end{remark}


\section{A functional CLT via Stein's method}\label{CLT sec}
Let $(\Om,\cF,P),\,\cF_{n,m}$, $\{\xi_n: n\geq 0\}$ and $F$ 
be as described in Section \ref{Sec1} and
Consider the random function $\cS_N$ given by 
\[
\cS_N(t)=N^{-\frac12}\sum_{n=1}^{[Nt]}F(\Xi_{n,N}),\,\,\text{ where }\,\,\Xi_{n,N}=(\xi_{q_1(n,N)},\xi_{q_2(n,N)},...,\xi_{q_\ell(n,N)})
\]
and assume that the limiting covariances $b(t,s)$ from Theorem \ref{VarThm} exist. For each $n,m$ and $N$ set
\[
d_N(n,m)=\min_{1\leq i,j\leq\ell}|q_i(n,N)-q_j(m,N)|.
\]

Set $l(N)=3N^{\zeta_1}+3$ where  $\zeta_1=\frac{3w}{4b+2\te(w-2)}$ and the numbers $w$ and $\te$ come from Assumptions \ref{Moment Ass} and \ref{Mix2}. Recall our assumption in Theorem \ref{CLTthm} that $\te>\frac{4w}{w-2}$, which implies that 
\begin{equation}\label{xi 1 props}
\big(\te(1-\frac 2w)\big)^{-1}<\zeta_1<\frac14.
\end{equation}
The proceeding arguments will be true for any $\zeta_1$ satisfying (\ref{xi 1 props}).
Define $r(N)=[l(N)/3]$ and for each $m$ and $r$ set 
\[
\xi_{m,r}=\bbE[\xi_m|\cF_{m-r,m+r}]\,\,\text{ and }\,\,\Xi_{m,N,r}=(\xi_{q_1(n,N),r},\xi_{q_2(n,N),r},...,\xi_{q_\ell(n,N),r}).
\]
We also set 
\begin{equation*}
\cZ_{N}(t)=N^{-\frac12}\sum_{n=1}^{[Nt]}F(\Xi_{n,N,r(N)})\,\,\text{ and }\,\,
\cW_{N}(t)=\cZ_{N}(t)-\bbE \cZ_{N}(t).
\end{equation*}
Using (\ref{F Hold}) and Assumption \ref{Mix2}, we can approximate
approximate $F(\Xi_{n,N})$ by $F(\Xi_{n,N,r(N)})$.
Taking into account Lemmas \ref{VarLemma}  and \ref{expectation lemma}, this yields that 
\begin{equation}\label{E}
\lim_{N\to\infty}\bbE[\cW_N(t)\cW_N(s)]=\lim_{N\to\infty}\bbE[\cS_N(t)\cS_N(s)]=b(t,s)
\end{equation}
and that the weak convergence of $\cS_N(\cdot)$ follows from the weak convergence of $\cW_N(\cdot)$. 
 We refer the readers to the beginning of the proof of Theorem 1.6.2 from \cite{book} for the precise details.

In the rest of this section we will prove the weak invariance principle for  $\cW_N(\cdot)$ using a functional version of  Stein's method.
For each $N$  we define a graph $(V_N,E_N)$ on $V_N=\{1,...,N\}$ by declaring that $(n,m)\in E_N$ if 
\[
d_N(n,m)=\min_{1\leq i,j\leq\ell}|q_i(n,N)-q_j(m,N)|\leq l(N)=3N^{\zeta_1}+3.
\]
Then the size of a ball of radius $1$ in this graph does not exceed $K_1l(N)$ for some constant $K_1$ which does not depend on $N$. Indeed let $n\in V_N$ and let  $m$ be a neighbor of $n$. Then $d_N(n,m)\leq 3N^{\zeta_1}+3$ and therefore there exists $1\leq i,j\leq\ell$ so that 
\[
q_i(n,N)-q_j(m,N)\in\{-l(N),-l(N)+1,...0,...,l(N)-1,l(N)\}:=\Gam_N.
\]
Hence $m$ solves one of the equations
\[
q_i(n,N)-q_j(m,N)=g,\,g\in\Gamma_N.
\]
Each equation has at most $d_j=\deg q_j$ solutions, $m_{g,1},...,m_{g,u_g}$, where $u_g\leq d_j$. Therefore, the unit ball around $n$ is contained in the set 
\[
\{n\}\cup\bigcup_{g\in\gamma}\{m_{g,1},...,m_{g,u_g}\}
\]
whose cardinality does not exceed $2d^*(l(N)+1)+1$, where $d^*$ is the maximal degree of the polynomials $q_1,...,q_\ell$.

Now, for each $N$ and $n\in V_N$ set
\[
X_{n,N}=\frac{F(\Xi_{n,N,r})-\bbE F(\Xi_{n,N,r(N)})}{\sqrt N}.
\]
Then 
$
\cW_N(t)=\sum_{n=1}^{[Nt]}X_{n,N}.
$
We also set 
\[
b_N(t,s)=\bbE[\cW_N(t)\cW_N(s)]=\text{Cov}(\cW_N(t),\cW_N(s)).
\]

Next, for any two sub-$\sigma$-algebras $\cG,\cH$ of $\cF$ we will measure the dependence 
between $\cG$ and $\cH$ via the quantities
\begin{equation*}
\ve_{p,q}(\cG,\cH)
=\sup\big\{|\text{Cov}(g,h)|:\,\max(
\|g\|_{L^p}, \|h\|_{L^q})\leq 1\big\}
\end{equation*}
where in the definition of $\ve_{p,q}(\cG,\cH)$ we consider function $g$ and $h$ so that $g$ is measureable with respect to $\cG$ and $h$ is measureable with respect to $\cH$. We recall the following classical relations between $\alpha(\cG,\cH)$ (defined in \ref{general alpha mixing}) and $\ve_{p,q}(\cG,\cH)$ (see, for instance, Theorem A.5, Corollary A.1 and Corollary A.2
in \cite{HallHyde}): for any $p>1$, we have
\begin{eqnarray}\label{alRel.0}
\ve_{p,\infty}(\cG,\cH)\leq 6\big(\al(\cG,\cH)\big)^{1-\frac1p}\,\,
\text{ and }\,\,\ve_{p,q}(\cG,\cH)\leq 8\big(\al(\cG,\cH)\big)^{1-\frac1p-\frac 1q}
\nonumber
\end{eqnarray}
for any  $q>1$ such that $\frac 1p+\frac 1q<1$. 

Now, for any $A\subset V_N$, let $\cG_A$ be the $\sig$-algebra generated by the random vector $X_A=\{X_{n,N}:\,n\in A\}$.
For any $A,B\subset V_N$ and $1\leq p,q\leq\infty$ we will measure the dependence 
between $X_A$ and $X_B$ via the quantities
\begin{equation*}
\alpha(A,B):=\alpha(\cG_A,\cG_B)\,\,\text{ and }\,\,
\ve_{p,q}(A,B)
:=\ve_{p,q}(\cG_A,\cG_B).
\end{equation*}
Then by (\ref{alRel.0}) for any $A,B\subset V_N$ and $p>1$,
\begin{eqnarray}\label{alRel}
\ve_{p,\infty}(A,B)\leq 6\big(\al(\cG_A,\cG_B)\big)^{1-\frac1p}\,\,
\text{ and }\,\,\ve_{p,q}(A,B)\leq 8\big(\al(\cG_A,\cG_B)\big)^{1-\frac1p-\frac 1q}
\end{eqnarray}
for any  $q>1$ such that $\frac 1p+\frac 1q<1$. 

Now, for reader's convenience, we will restate Theorem 1.5.1 from \cite{book}:

\begin{theorem}\label{ThmFunc} 
Let  $p_0,q_0\geq1$ and set
\begin{equation}\label{d_M0}
\tau_N=d_1+d_2+d_3+d_4
\end{equation} 
where with $X_n=X_{n,N}$, $\sig_{n,m}=\bbE X_nX_m$,  $N_n=\{n\}\bigcup\{m:\,(n,m)\in E_N\}$ and $N_n^c=V_N\setminus N_n$,
\begin{eqnarray}
d_1=\sum_{n=1}^N\|X_n\|_{L^{p_0}}\ve_{p_0,\infty}(\{n\},N_n^c)\label{d1},\\
d_2=\sum_{n=1}^N\sum_{m\in N_n}
\|X_nX_m-\sig_{n,m}\|_{L^{q_0}}\ve_{q_0,\infty}
(\{n,m\},N_n^c\cap N_m^c),\label{d2}\\
d_3=\sum_{n=1}^N\sum_{m,k\in N_n}\big(
\bbE |X_nX_mX_k|+\bbE |X_nX_m|\bbE |X_k|\big),\label{d3}\\
d_4=\sum_{n=1}^N\sum_{m\in N_n^c}
\|X_n\|_{L^{p_0}}\|X_m\|_{L^{p_0}}\ve_{p_0,p_0}(\{n\},\{m\})\label{d4}.
\end{eqnarray} 
Suppose that there exists $\Gam>0$ such that 
\begin{equation}\label{UnifInt}
\|\cW_N(s)-\cW_N(t)\|_{L^2}^2\leq\Gam\frac{[Ns]-[Nt]}N
\end{equation}
for any $N\in\bbN$ and $0\leq t\leq s\leq 1$. Furthermore assume that the limits
$\lim_{N\to\infty} b_N(t,s)=b(t,s), s,t\in[0,1]$ exist and that 
\begin{equation}\label{tau lim}
\lim_{N\to\infty}\tau_N\ln^2N=0.
\end{equation}
Then
$\cW_N$  weakly converge in the Skorokhod space $D$ to a
 continuous centered Gaussian process $G$ with covariance function $b(\cdot,\cdot)$. 
\end{theorem}
Theorem \ref{ThmFunc} is essentialy due to A.D. Barbour, and it follows from the arguments in \cite{Barb}
and \cite{Barb1} (see the proof of Theorem 1.5.1 in \cite{book}). 
 
In the rest of this section we will show that all the conditions of Theorem \ref{ThmFunc} are satisfied with
\[
p_0=2q_0=w
\] 
where $w$ comes from Assumption \ref{Moment Ass}. Then there exists a constant $C>0$ so that for any $n$ we have
\begin{equation}\label{X n bound}
\|X_n\|_w\leq CN^{-\frac12}.
\end{equation}
That the covariances converge we have already shown.
We claim next that there exists $C>0$ such that 
\begin{equation}\label{1-st}
\text{Var}(\cW_N(s)-\cW_N(t))\leq C^2\frac{[Ns]-[Nt]}N
\end{equation}
for any $N\in\bbN$ and $0\leq t\leq s\leq 1$. The arguments proceeding (\ref{E}) yield that
it is enough to prove that (\ref{1-st}) holds true with $\cS_N$ in place of $\cW_N$.  As in (\ref{E}), we refer the readers to the proof of Theorem 1.6.2 in \cite{book} for the exact technical details.
By Proposition \ref{1.3.14}
there exists a constant $c_0>0$ such that 
for any $n,m\in\bbN$,
\[
|\text{Cov} (F(\Xi_{n,N}),F(\Xi_{m,N}))|\leq c_0\tau([\frac13 d_N(n,m)])
\]
where $\tau(k)=\phi_k^{1-\frac1w}+\be_{q,k}^\ka$ Observe that by Assumption \ref{Mix2} we have $\tau(k)\leq dk^{-(1-\frac1w)\te},\,k\geq 0$. The assumptions in Theorem \ref{CLTthm} imply that $\te(1-\frac1b)>1$, and therefore
\[
\sum_{k=0}^\infty\tau(k)<\infty.
\]
Let $d^*$ be the maximal of the polynomials $q_i$. Then
for any $k\geq 0$ and $n\in\bbN$
there exist at most $2\ell^2d^*$ natural $m$'s such that $d_N(n,m)=k$.  
Hence,  for any $n\in\bbN$ we have
\[
\sum_{m=1}^\infty|\text{Cov} (F(\Xi_{n,N}),F(\Xi_{m,N}))|\leq 6\ell^2d^*c_0\sum_{s=0}^\infty\tau(s):=A<\infty
\]
for any natural $n$, which implies that
\begin{equation*}
\text{Var}(\cS_N(s)-\cS_N(t))
\leq N^{-1}\sum_{n=[Nt]}^{[Ns]}\sum_{m=[Nt]}^{[Ns]}|\text{Cov} (F(\Xi_{n,N}),F(\Xi_{m,N}))|
\leq A\frac{[Ns]-[Nt]}N
\end{equation*}
implying  (\ref{1-st}) with $\cS_N$ in place of $\cW_N$. In that above arguments we have used Assumption \ref{Mix2} (and that  $\te(1-\frac1w)>1$) in order to derive that $A$ is indeed finite. 

It remains to show that $\lim_{N\to\infty}\tau_N\ln^2N=0$. We begin with estimating $d_1$ and $d_4$.  Let $n\in\bbN$ and set 
\[
A=\{n\}\,\,\text{ and } B=N_n^c.
\]
 Let us estimate 
$
\alpha(A,B).
$
Consider the sets 
\[
\Gam_1=V=\{q_{i}(n,N):\,1\leq i\leq\ell\},\,\,\text{ and }\,\,\Gam_2=\{q_{i}(m,N):\,1\leq i\leq\ell,\,m\in N_n^c\}.
\]
Then for each $\gam_1\in\Gam_1$ and $\gam_1\in\Gam_2$ we have
\[
|\gam_1-\gam_2|>l
\]
and therefore, there exist disjoint sets $Q_1,...,Q_L$,\,$L\leq 2\ell+1$ so that
\[
\Gam_1\cup\Gam_2=\bigcup_{i=1}^L Q_i,
\]
 each one of the $Q_i$'s is contained either in $\Gam_1$ or in $\Gam_2$ 
and
\[
q_i+r\leq q_{i+1}
\]
for any $1\leq i\leq L-1$, $q_i\in Q_i$ and $q_{i+1}\in Q_{i+1}$.  Consider the random vector $U=(U_1,...,U_L)$  where for each $i$,
\[
U_i=\{\xi_{m,r(N)}:\,m\in Q_i\}.
\]
Let $\{\cC_1,\cC_2\}$ be the partition of $\Gam_1\cup\Gam_2$ given by 
\[
\cC_i=\bigcup_{j:Q_j\subset\Gam_i}Q_j,\,i=1,2.
\]
Then $\alpha(A,B)=\alpha\big(\sig\{U(\cC_1),U(\cC_2)\}\big)$ and by Proposition \ref{1.3.11} we have
\[
\alpha(A,B)=\alpha\big(\sig\{U(\cC_1),U(\cC_2)\}\big)\leq L\phi(r(N))
\]
where $U(\cC_i)=\{U_j:\,j\in\cC_i\}$ and $\sig\{U(\cC_i)\}$ is the $\sig$-algebra generated by $U(\cC_i)$, $i=1,2$. We conclude from (\ref{alRel}), (\ref{d1}), (\ref{d4}) and (\ref{X n bound}) that there exists a constants $C_1$ and $C_4$ so that 
\begin{equation}\label{d1 d4 est}
d_1\leq C_1 N^{\frac12}\big(\phi(r(N))\big)^{1-\frac 1w}\,\,\text{ and }\,\,d_4\leq C_4 N\big(\phi(r(N))\big)^{1-\frac 2w}.
\end{equation}

Next, we will estimate $d_2$. Let $n,m\in V_N$ be so that $m\in N_n$. Set 
\[
A=\{n,m\}\,\text{ and }\,\,B=N_n^c\cap N_m^c.
\]
Let use first estimate $\alpha(A,B)$. 
Consider the sets $\Gam_1=\{q_{i}(s,N):\,1\leq i\leq\ell,\,s=n,m\}$ and
\[
\Gam_2=\{q_{i}(m,N):\,1\leq i\leq\ell,\,m\in N_n^c\cap N_m^c\}.
\]
Then for any $\gam_1\in\Gam_1$ and $\gam_2\in\Gam_2$ we have
\[
|\gam_1-\gam_2|>l
\]
and therefore, there exist disjoint sets $Q_1,...,Q_L$,\,$L\leq 4\ell+1$ so that
\[
\Gam_1\cup\Gam_2=\bigcup_{i=1}^L Q_i,
\]
 each one of the $Q_i$'s is contained either in $\Gam_1$ or in $\Gam_2$ 
and
\[
q_i+r(N)\leq q_{i+1}
\]
for any $1\leq i\leq L-1$, $q_i\in Q_i$ and $q_{i+1}\in Q_{i+1}$. Using this partition, exactly as in the estimates of $d_1$ and $d_4$ we obtain  that
\begin{equation}\label{alpha d1}
\alpha(A,B)\leq L\phi(r(N)).
\end{equation} 
We conclude from (\ref{alRel}), (\ref{d2}) and  (\ref{X n bound}) that there exists a constant $C_2$ so that  
\begin{equation}\label{d2 est}
d_2\leq C_2 N\big(\phi(r(N))\big)^{1-\frac 2w}.
\end{equation}

Finally, by (\ref{X n bound}) and the H\"older inequality, each one of the summands in $d_3$ does not exceed $CN^{-\frac32}$ for some $C>0$ and therefore 
\begin{equation}\label{d3 est}
d_3\leq C_3N^{-\frac12}(l(N))^2=C_3N^{-\frac12+2\zeta_1}
\end{equation}
where we used that $|N_n|\leq K_1 l(N)$ for some $K_1>0$ and any $n$ and $N$. Relying on (\ref{xi 1 props}), (\ref{d1 d4 est}), (\ref{d2 est}), (\ref{d3 est})
and on the inequality
\[
\phi(r(N))\leq d(r(N))^{-\theta}\leq dN^{-\te \zeta_1}
\]
where $d$ and $\te$ come from Assumption \ref{Mix2},
we conclude that (\ref{tau lim}) holds true
and the proof of Theorem \ref{CLTthm} is complete.


\section{Applications}\label{Applications}
In this section we will describe several type of processes $\{\xi_n\}$ fro which all the results stated in Section \ref{Sec1} hold true.
\subsection{Hidden Markov chains and related processes}
Let $\cX$ be a topological space and let $\cB$ be the space of all bounded measurable functions on $\cX$, equipped with the supremum norm $\|\cdot\|_\infty$. Let $R:\cB\to\cB$ be a positive operator so that $R\textbf{1}=\textbf{1}$, where $\textbf{1}$ is the function which takes the constant value $1$ (i.e. $R$ is a Markov-Feller operator). 
We assume here that $R$ has a stationary probability measure $\mu$ so that  for any $n\geq1$ and $g\in\cB$ we have
\begin{equation}\label{Conv}
\|R^ng-\mu(g)\|_\infty\leq \|g\|_\infty\tau(n)
\end{equation}
for some sequence $\tau(n)$ which converges to $0$ as $n\to\infty$.
 Let $\{\Upsilon_n\}$ be the stationary Markov chain with initial  distribution $\mu$, whose transition operator is $R$. Then the inequality (\ref{Conv}) holds true with $\tau(n)$ of the form $\tau(n)=Ce^{-cn},\,c,C>0$ for an aperiodic Markov chain $\{\Upsilon_n\}$ if, for instance, a version
of the Doeblin condition holds true (see, for instance, Section 21.23 in \cite{Br}).

Next, for any $0\leq n\leq m$ let $\cF_{n,m}=\sig\{\Upsilon_n,...,\Upsilon_m\}$ be the $\sig$-algebra generated by the random variables $\Upsilon_n,\Upsilon_{n+1},...,\Upsilon_m$. When $n$ is negative we set $\cF_{n,m}=\cF_{0,\max(0,m)}$.
The following result is well known (see \cite{Br}), but it has a short proof which is given  here for readers' convenience. 

\begin{lemma}\label{MC mix lemma}
Suppose that (\ref{Conv}) holds true.
Then the mixing coefficients $\phi(n)$ corresponding to the $\sig$-algebras $\cF_{n,m}$ defined above satisfy $\phi(n)\leq 2\tau(n)$.
\end{lemma}
\begin{proof}
First (see \cite{Br}, Ch. 4), for any two sub-$\sigma-$algebras $\cG,\cH\subset\cF$ we have
\begin{equation}\label{Phi-Rel-StPaper}
\phi(\cG,\cH)=\frac12\sup\{\|\bbE[h|\cG]-\bbE h\|_{\infty}\,:  h\in L^\infty(\Om,\cH,P),\,\|h\|_{\infty}\leq1\}
\end{equation}
where $\phi(\cG,\cH)$ is defined in \ref{general phi mixing} (so $\phi(n)$ is given by (\ref{MixCoef1})).
Let $k$ and $n$ be nonnegative integers, and set 
\[
\cG=\sig\{\Upsilon_0,\Upsilon_1,...\Upsilon_k\},\,\,\text{ and }\,\,\cH=\sig\{\Upsilon_m:\,m\geq k+n\}.
\]
Let $h$ be a $\cH$-measurable random variable which is bounded $P$-a.s. by $1$. Then we can write $h=H(\Upsilon_{k+n},\Upsilon_{k+n+1},...)$ for some measurable function $H$ so that $|H|\leq1$. Set $H_{n+k}=\bbE[h|\Upsilon_0,\Upsilon_1,...,\Upsilon_{n+k}]$. Then $H_{n+k}=H_{n+k}(\Upsilon_{n+k})$ is a function of $\Upsilon_{n+k}$, and 
\[
\bbE[h|\cG]=\bbE[H_{n+k}(\Upsilon_{n+k})|\cG]=R^{n}H_{n+k}(\Upsilon_k)
\]
where we used that $\sig\{\Upsilon_0,\Upsilon_1,...,\Upsilon_{n+k}\}$ is finer than $\cG$ (and the tower property of conditional expectations).
Using (\ref{Conv}) and taking into account that  $|H_{n+k}|\leq1$ and that 
$\bbE h=\bbE H_{n+k}(\Upsilon_{n+k})=\mu_{n+k}(H_{n+k})$ we obtain that 
\[
\|\bbE[h|\cG]-\bbE h\|_{L^\infty}\leq \tau(n).
\]
Taking the supremum over all the above functions $h$ and using (\ref{Phi-Rel-StPaper}) we obtain that 
\[
\phi(\cG,\cH)\leq 2\tau(n).
\]
Taking the supremum over all choices of $k$ completes the proof of the lemma.
\end{proof}

Let  $f=(f_1,...,f_d)$ is a measurable bounded function and for each $n\geq 0$ set $\xi_n=f(\Upsilon_n)$. Then, in the notations of Section \ref{Sec1} we have  $\be_q(r)=0$ for any $r\geq0$ and $q$. Therefore, all he results stated there hold true with the stationary sequence $\{\xi_n\}$ when $\tau(n)$ satisfies $\tau(n)\leq dn^{-\theta}$ for some $d>0$ and $\theta>0$. When the random variables $\xi_n$ above only belong to $L^w$ for some $w>2$ then we our results hold true when 
 $\tau(n)\leq dn^{-\theta}$ for some $d>0$ and $\theta>\frac{4w}{w-2}$. In particular we can consider the case when $\tau(n) $ converges exponentially fast to $0$. 
Suppose now that  $(\cX,\rho)$ is a metric space. Let $\rho_\infty$ be the metric on $\cX^\bbN$  given by 
\[
\rho_\infty(x,y)=\sum_{j=1}^\infty 2^{-j}\frac{\rho(x_j,y_j)}{1+\rho(x_j,y_j)},\,\, x=(x_j),\, y=(y_j).
\]
Let $f:\cX^\bbN\to\bbR^d$ be a H\"older continuous function with respect to the metric $\rho_\infty$, and for each $n\geq0$ set  $\xi_n=f(\Upsilon_n,\Upsilon_{n+1},\Upsilon_{n+2},...)$. In these circumstances, it is clear that 
\[
\beta_\infty(r)=\sup_{n\geq 1}\|\bbE[\xi_n|\cF_{n-r,n+r}\|_{L^\infty}\leq A\del^r
\]
for some constants $A>0$ and $\del\in(0,1)$. Hence, $\tau(n)\leq dn^{-\theta}$, then all the results stated in Section \ref{Sec1} hold true for the stationary sequence $\{\xi_n\}$ defined above. Here $\theta>\frac{4w}{w-2}$, $d>0$ and $w>2$ satisfies that $\|\xi_n\|_{L^w}<\infty$ (assuming that such a $w$ exists).

We can also consider several types of linear Markov processes, described in what follows. Let $\{\Upsilon_n:\,n\in\bbZ\}$ be a two sided stationary Markov chain with transition operator $R$ and stationary distribution $\mu$. Then Lemma \ref{MC mix lemma} also holds true with the $\sig-$algebras $\cF_{n,m}=\sig\{\Upsilon_n,...,\Upsilon_m\}$. Let $(a_n)$ be a two sided sequence such that $\sum_{n\in\bbZ}|a_n|<\infty$ and let $f:\cX\to\bbR$ be a bounded function so that $\int f d\mu=0$. For each $i$ set
\[
\xi_i=\sum_{n\in\bbZ}a_{n}f(\Upsilon_{n+i}).
\] 
Then $\{\xi_i:\,i\in\bbZ\}$ is a bounded stationary sequence of random variables.
Observe that for each $n$ and $k\geq0$,
\begin{eqnarray}\label{f cor }
|\bbE[f(\Upsilon_n)f(\Upsilon_{n+k})]|=|\bbE[R^kf(Y_n)f(Y_n)]|\\\leq 
\|f\|_\infty\|R^k f\|_\infty=\|f\|_\infty\|R^k f-\mu(f)\|_\infty
\leq \|f\|_\infty^2\tau(k)\nonumber
\end{eqnarray}
where in the last inequality we used (\ref{Conv}).
Suppose that $\sum_{n=0}^\infty\tau(n)<\infty$. Then,
using (\ref{f cor }), a direct calculation shows that for any $r>0$,
\[
\Big\|\xi_i-\sum_{|n|\leq r}a_{n}f(\Upsilon_{n+i})\Big\|_{L^2}\leq C\sum_{|n|>r}|a_n|
\]
where $C$ is some constant. Therefore, the approximation coefficients $\beta_2(r)$ defined in Section \ref{Sec1} satisfy 
\[
\beta_2(r)\leq C\sum_{|n|>r}|a_n|.
\]
Thus, when $\sum_{|n|>r}|a_n|$ and $\tau(r)$ converge to $0$ sufficiently fast as $r\to\infty$, all of the results stated in Section \ref{Sec1} (with $q=2$) hold true with the stationary sequence $\{\xi_i:\,i\geq0\}$.

\subsection{Subshifts of finite type (and uniformly hyperbolic and distance expanding maps) and continued fraction expansions}
Next, we recall the definition of a (topologically mixing) subshift of finite type. Let $d>1$ be a positive integer and set $\cA=\{1,2,...,d\}$. We consider here $\cA$ as a discrete topological space, and let $\cX=\cA^{\bbN\cup\{0\}}$ be the product (topological) space. We define a metric on $\cX$ by 
\[
d(x,y)=2^{-\inf\{n:\,x_n\not=y_n\}}
\]
where we set $\inf\emptyset=\infty$ and $2^{-\infty}=0$.
Then the product topology is generated by this metric.
Let $A=(A_{i,j})$ be a $d\times d$ matrix with $0-1$ entries. Suppose  that $(A^M)_{i,j}>0$ for some $M$ and all $1\leq i,j\leq d$, and set 
\[
\Sigma(A)=\{(x_i)\in\cX:\,A_{x_i,x_{i+1}}=1\,\,\forall\,i\geq0\}.
\]
Let $T:\Sigma(A)\to\Sigma(A)$ be the left shift given by 
\[
(Tx)_i=x_{i+1},\,i\geq0.
\]
Then $\Sig(A)$ is $T$ invariant.
Let $\mu$ be any invariant Gibbs measure (see \cite{Bow}). For each finite word $(a_0,...,a_r)\in\cA^{r+1}$ we define its corresponding cylinder set $[a_0,...,a_r]$ to be the set of all $x\in\Sigma(A)$ so that $x_i=a_i$ for any $1=0,1,...,r$. The length of such a set is defined to be $r+1$.
Let
$\cF_{0,n}$ be the $\sig$-algebra generated by all cylinder sets of length $n$ and for each $0\leq n\leq m$ set $\cF_{n,m}=T^{-n}\cF_{0,m-n}$. When $n$ is negative we set $\cF_{n,m}=\cF_{0,\max(0,m)}$.
Then (see \cite{Bow}), these $\sig$-algebras satisfy that $\phi(n)\leq A\del^n$ for some $A>0$ and $\del\in(0,1)$ (in fact, we also have exponentially fast $\psi$-mixing).

For each $n\geq 0$ set $\xi_n(x)=T^nf(x)$, where $f=(f_1,...,f_d)$ is an $\bbR^d$-valued function, each $f_i$ is a H\" older continuous function, and $x$ is distributed according to $\mu$. Then 
\[
\beta_\infty(r)=\sup_{n\geq0}\|\xi_n-\bbE[\xi_n|\cF_{n-r,n+r}]\|_{L^\infty}\leq Ca^r
\]
for some $C>0$ and $a\in(0,1)$. Note that when $f$ is constant on cylinder sets (and hence H\"older continuous) then $\be_\infty(r)=0$ for any sufficiently large $r$.
 We conclude that all the results stated in Section \ref{Sec1} hold true for the sequences $\{\xi_n\}$ defined above.
Using \cite{Bow}, we obtain that the results from Section \ref{Sec1} also hold in the case when 
$\xi_n=T^nf$, where $f=(f_1,...,f_d)$, $T$ is a hyperbolic diffeomorphism or an expanding transformation
taken with a Gibbs invariant measure, and
 each $f_i$ is either H\" older
continuous or piecewise constant on elements of Markov partitions.

Next, set $\cX=(0,1)\setminus\bbQ$, let $T:\cX\to\cX$ be the Gauss map which is given by $Tx=\frac 1x-[\frac 1x]$, and let $\mu$ be the unique absolutely continuous $T$-invariant probability measure given by $\mu(A)=\frac1{\ln 2}\int_{A}\frac1{1+x}dx$.
Let $\cA$ be the partition of $\cX$ into the intervals $I_n=(\frac 1{n+1},\frac 1n)$, where $n=1,2,...$. For each $i=0,1,2,...$ let $n_i(x)$ be the unique positive integer so that $T^{i}x\in I_{n_i(x)}$.
Then the map $x\to(n_0(x),n_1(x),...)$ represents the continued fraction expansion of $x$.
Set $\cF_{n,m}=\bigvee_{j=n_0}^{m_0}T^{-j}\cA$, where $n_0=\max(0,n)$ and $m_0=\max(0,m)$. Then these $\sigma-$algebras are exponentially fast $\psi$-mixing, and, in particular $\phi(n)\leq A\del^n$ for some $A>0$, $\del\in(0,1)$ and all nonnegative integers $n$. Moreover, the partition $\cF_{0,n}$ is a partition into intervals whose lengths do not exceed $Ce^{-cn}$  for some constants $c,C>0$. Therefore, all the conditions from Section \ref{Sec1} also holds true for stationary sequences of the form $\xi_n=f\circ T^n(x)$, where $x$ is distributed according to $\mu$ and $f:[0,1]\to\bbR^d$ is either a H\"older continuous function or a function which is constants of the elements of the partition $\cF_{0,r}$ for some fixed $r$.

\subsection{Extension to Young towers}\label{YoungSec}
Let $(\Del,\nu,T)$ be the noninvertible and mixing Young tower considered in
\cite{Young2} (or the projected tower considered in Section 3 of \cite{Young1}). Let $\Del_0$ be the base of the tower, $R:\Del_0\to\Del_0$ be the return time function and let $d(x,r)=\be^{s(x,y)},\,\beta\in(0,1)$ be the dynamical distance 
defined by the separation time $s(x,y)$ from \cite{Young2} (or the one on the projected tower in \cite{Young1}). In this section we will denote the levels of the tower by $\Del_\ell,\,\ell\geq0$ where $\Del_0$ is identified with $\Del_0\times\{0\}$ and for each $\ell>0$,
\[
\Del_\ell=\{(x,\ell):\,x\in\Del_0\,\,R(x)>\ell\}.
\]
Let $\cL$ be the transfer operator associated with the tower $T$ (so $\nu$ is its  conformal measure- the eigen-measure of $\cL$) and $h$ be the eigenfunction of $\cL$ (see \cite{Young1} and \cite{Young2}). Let us denote by $\cA^n$ the $\sig$-algebra  generated by all cylinder sets of length $n$, and
let  the $\sig-$algebras $\cF_{n,m}$ be given by $\cF_{n,m}=T^{-n}\cA^{m-n}$, $0\leq n\leq m$ while when $n<0$ we set $\cF_{n,m}=\cF_{0,\max(0,m)}$.

 We will consider here processes of the form $\xi_n(x)=f\circ T^n(x)$ where $f:\Del\to\bbR^d$ is a H\"older continuous function so that $\|f\|_{L^{2q}(\nu)}<\infty$ for some $q\geq 1$, and $x$ is distributed according to the absolutely continuous invariant measure $\mu$ given by $d\mu=hd\nu$.  In this case, it is clear that 
\begin{equation}\label{YT approx}
\be_\infty(r)=\sup_{n\geq0}\|\xi_n-\bbE[\xi_n|\cF_{n-r,n+r}]\|_{L^\infty}\leq Ce^{-cr}
\end{equation}
for some $C,c>0$, since we can approximate $f$ uniformly by functions which depend only on elements of the partition $\cA^r$, and $\bbE[\xi_n|\cF_{n-r,n+r}]=\bbE[f|\cA^r]\circ T^n$.

 The family of $\sig$-algebras $\cF_{n,m}$ does not seem to be $\phi$-mixing in the sense of Section \ref{Sec1}, and so we can not apply the results from Section \ref{Sec1}. Note that when the tails $\mu\{R>j\}$ decay polynomially sufficiently fast to $0$ then the map $T^R:\Del_0\to\Del_0$ is $\phi$-mixing (see Lemma 2.4 (b) in \cite{Melb}), in the sense that its corresponding family of $\sig-$algebras $\cF_{n,m}$ is (left) exponentially $\phi$-mixing. Relying on this we could probably extend the results from Section \ref{Sec1} under certain restriction on the behaviour of the nonconventional sums between two consecutive returns to the base $\Del_0$.
Still, we claim that all the results stated in Section \ref{Sec1} hold true for the above sequence $\{\xi_n\}$ when $\nu\{R>n\}\leq An^{-d}$ for some $A>0$ and a sufficiently large $d>0$, without restrictions of that kind. 

Using tower extensions, our results  also hold true in the case  when $\xi_n=f\circ T^n$, for several classes of map which satisfy a certain hyperbolicity or expansion conditions only on some peices of the underlying manifold (or, for any non-invertible dynamical systems that can be modeled by a Young tower). Results in the invertible case also follow (in the exponential tails case), by considering first the projected tower (see Section 3 in \cite{Young1}) and then proceeding essentially as in Section 4 in \cite{Young1} in order to derive Theorems \ref{VarThm} and \ref{CLTthm} for the original system from the corresponding limit theorems on the projected tower. We refer the readers to \cite{Young1}, \cite{Young2}, \cite{Melb} and \cite{Hyd} for examples of  maps $T$ which can be modelled by towers.

In the rest of the section  we will explain how to obtain the results from Section \ref{Sec1} in the above Young tower setup. First, 
let $v>0$ be a function which is constant on the levels $\Del_\ell$ of the tower and define a transfer operator $L$ by
\[
Lf=\frac{\cL(fv)}{v}.
\]
Then the measure $\nu_L$ given by $d\nu_L=vd\nu$ is conformal with respect to $L$ and the function $h_L=h/v$ is preserved under $L$. For any measuable set $A$,  let us denote by $\bbI_A$ its indicator function. We will rely on

\begin{lemma}\label{tower mix lemma}
There exists a constant $C>0$ so that
for any $n\geq0$, $\,A\in\cA^n$ and an arbitrary measurable set $B$,
\begin{eqnarray*}
|\nu\big(h\bbI_A\cdot \bbI_B\circ T^{k+n}\big)-\nu(h\bbI_A)\nu(h\bbI_B)|\leq 
C\nu_L(B)\sup_{g: \|g\|=1}\|L^k(g/v)-\nu_L(g/v)h_L\|_\infty\\=
C\nu_L(B)\sup_{g: \|g\|=1}\|\cL^k(g)/v-\nu(g)h/v\|_\infty.
\end{eqnarray*}
Here $\|\cdot\|_\infty$ stands for the supremum norm and $\|g\|:=\|g\|_\infty+K_g$, where $K_g$ is the infimum of the set of values $K$ so that $|g(x)-g(y)|\leq Kd(x,y)$ for any $x,y$ in the same level of the tower.
\end{lemma}

Before proving this lemma we will first explain how it will be used.
In the setup of Section 3 of \cite{Young1} (the projected tower setup), for some $v$ so that $v_\ell:=v|\Del_\ell=e^{\ve\ell}$ (where $\Del_\ell$ is the $\ell$-th floor and $\ve>0$ is some constant), in Section 3 of \cite{Young1} (Proposition A) it was proved that 
\[
\sup_{g: \|g\|=1}\|\cL^k(g)/v-\nu(g)h/v\|\leq A\del^k
\]
for some $A>0$ and $\del\in(0,1)$. In the setup of \cite{Young2},  when $\nu\{R>j\}\leq Cj^{-d}$ for some $C>0$ and $d>2$ then 
by taking $v_\ell=\ell^{d-2}$ (so that $\sum_{\ell}v_\ell\nu(R>\ell)<\infty$) we obtain from
Proposition 3.13 in \cite{Viv} that
\[
\sup_{g: \|g\|=1}\|L^n(g/v)-\nu_L(g/v)h_L\|_\infty\leq An^{-(d-2)}
\]
for some constant $n$. In fact, also the exponential case is considered in \cite{Viv} and it is possible also to get the same estimates with the norm $\|\cdot\|$ in place of the supremum norm.

When $v|\Del_\ell=e^{\ve\ell}$ and $B$ is contained in a the union of the first $j$ floors we get that
\[
\nu_L(B)=\nu(v\cdot\bbI_B)\leq e^{\ve j}\nu(B).
\]
Hence, there exists a constant $c>0$ so that for any $j$ and $k$ satisfying
$j\leq ck$ we have
\begin{equation}\label{MixAppl}
\phi(T^{-(n+k)}\cG_j,\cA^n)\leq Ae^{-ak}
\end{equation}
where $A,a>0$ and $\cG_j$ is the induced $\sig$-algebra on the union of the first $j$ floors.  Similarly, when $v_\ell=\ell^{d-2}$, then for any $j, k$ and $\al\in(0,1)$ so that  $j\leq k^\alpha$ we have
\begin{equation}\label{MixApp2}
\phi(T^{-(n+k)}\cG_j,\cA^n)\leq A k^{-(d-2)(1-\alpha)}
\end{equation}
for some constant $A>0$. We will show after the proof of Lemma \ref{tower mix lemma} how to use (\ref{MixApp2}) in order to derive all the results stated in Section \ref{Sec1} in the Young tower case.

\begin{proof}[Proof of the Lemma \ref{tower mix lemma}]
Let $n,A$ and $B$ be as in the statement of the lemma.
Write 
\begin{eqnarray*}
\nu\big(h\bbI_A\cdot \bbI_B\circ T^{k+n}\big)-\nu(h\bbI_A)\nu(h\bbI_B)=
\nu_L\big(h_L\bbI_A\cdot \bbI_B\circ T^{k+n}\big)-\nu_L(h_L\bbI_A)\nu_L(h_L\bbI_B)\\=
\nu_L\big(L^{k+n}(h_L\bbI_A)\cdot \bbI_B\big)-\nu_L(h_L\bbI_A)\nu_L(h_L\bbI_B)=
\int_B\big(L^k(f_n)-\nu_L(f_n)h_L\big)d\nu_L
\end{eqnarray*}
where 
\[
f_n=L^n(h_L\bbI_A)=\frac{\cL^n(h\bbI_A)}{v}.
\]
Observe that 
\[
\|vf_n\|_\infty\leq\|\cL^nh\|_\infty=\|h\|_\infty<\infty.
\]
Moreover, by Lemma 4 in \cite{Hyd}, for any $x,y$ in the same floor we have 
\[
|(vf_n)(x)-(vf_n)(y)|\leq c_1\cL^n(h\bbI_A)(y)d(x,y)\leq c_1\|h\|_\infty d(x,y):=C_1d(x,y)
\]
where $c_1$ is some constant which does not depend on $n$ and $A$
(note that similar estimates appear in the Sublemma at the beginning of Section 4.2 in \cite{Young1}).
We conclude that 
\begin{eqnarray*}
|\nu\big(h\bbI_A\cdot \bbI_B\circ T^{n+k}\big)-\nu(h\bbI_A)\nu(h\bbI_B)|\leq 
C\nu_L(B)\sup_{g: \|g\|=1}\|L^k(g/v)-\nu_L(g/v)h_L\|_\infty\\=
C\nu_L(B)\sup_{g: \|g\|=1}\|\cL^k(g)/v-\nu(g)h/v\|_\infty
\end{eqnarray*}
where $C$ is some constant.
\end{proof}

Next, we will explain how to use Lemma \ref{tower mix lemma} in order to obtain functional central limit theorems for nonconventional polynomial arrays in the case when $\xi_n(x)=f\circ T^n(x)$ discussed at the beginning of this section.
For any $j\geq 0$, let $\Del^{(j)}$ be the union of the first $j$ floors, and let $\chi_j$ be its indicator set. Then there exists a constant $C>0$ so that for any $r\geq0$,
\begin{eqnarray*}
\|\bbE[f|\cA^r](1-\chi_j)\|_q\leq \|f\|_{2q}\big(\mu\{R>j\}\big)^{\frac 1{2q}}\\\leq 
\|g\|_{2q}\cdot(\|h\|_\infty)^{\frac 1{2q}}\big(\nu\{R>j\}\big)^{\frac 1{2q}}
\leq Cj^{-d/2q}
\end{eqnarray*}
where we assumed that the tails decay at least as fast as $j^{-d}$, that $\|f\|_{2q}<\infty$ and we write $\|\cdot\|_q=\|\cdot\|_{L^q(\Del,\mu)}$.
Taking into account (\ref{YT approx}) with $n=0$
we conclude that
\begin{eqnarray*}\label{New approx}
\,\,\,\,\,\,\be_{q}(r,j):=\sup_{n\geq0}\|\xi_n-\big(\bbE[f|\cA^r]\cdot\chi_j\big)\circ T^n\|_{q}=
\sup_{n\geq0}\|(f-\bbE[f|\cA^r]\chi_j)\circ T^n\|_q\\=\|f-\bbE[f|\cA^r]\chi_j\|_q\leq 
\|f-\bbE[f|\cA^r]\|_q+\|\bbE[f|\cA^r](1-\chi_j)\|_q
\leq C_1(e^{-cr}+j^{-d/2q})\nonumber
\end{eqnarray*}
for some constant $C_1>0$. 
Using the approximation coefficients $\be_{q}(r,j)$ in place of $\be_{q}(r)$ from Section \ref{Sec1}, and Lemma \ref{tower mix lemma}, the proofs of all the results stated in Section \ref{Sec1} proceed essentially in the same way when $d/2q$ is sufficiently large. Indeed, all the results from Sections \ref{Sec3} and \ref{Var Sec} rely only on Proposition \ref{1.3.14} together with several combinatorial arguments. This corollary has an appropriate version which involves the approximation coefficients $\be_{q}(r,j)$ (instead of $\be_q(r)$), since Proposition \ref{1.3.11} can be derived also using the right $\phi$-mixing coefficients. Using this version of Proposition \ref{1.3.14}, we obtain all the results from Sections \ref{Sec3} and \ref{Var Sec} also in the Young tower case.
Relying on the above version of  Proposition \ref{1.3.11}, Theorem \ref{ThmFunc} can be applied successfully also in the Young tower case similarly to Section \ref{CLT sec}, using the seqeunce $\be_{q}(r(N),j(N))$ instead of $\be_q(r(N))$, where $r(N)=[l(N)/3]$ is the same as in Section \ref{CLT sec} and $j(N)=[\big(r(N))^\alpha]$ for a sufficiently small $\alpha\in(0,1)$. Note that in the appropriate applications of (\ref{MixApp2}) we have to take $k=r(N)$ and so, $j=j(N)$ will indeed satisfy $j\leq k^\alpha$.


 

\end{document}